\documentclass{amsart}

\usepackage{amssymb,latexsym,amsfonts,amsmath}
\usepackage{graphicx,psfrag,epsfig,epsf}
\include{diagrams}
\usepackage{eso-pic}
\usepackage{graphicx}
\usepackage{color}
\usepackage{diagrams}
\usepackage[boxed]{algorithm2e}
\usepackage{tikz}
\usetikzlibrary{automata}
\usepackage{subfigure}
\usepackage{amsfonts}

\usepackage{amssymb,latexsym,amsfonts,amsmath}
\usepackage{graphicx}
\usepackage{mathabx}

\def\baselinestretch{1}

\newtheorem{theorem}{Theorem}[section]

\newtheorem{problem}[theorem]{Problem}
\newtheorem{proposition}[theorem]{Proposition}
\newtheorem{corollary}[theorem]{Corollary}
\newtheorem{assumption}[theorem]{Assumption}

\theoremstyle{definition}
\newtheorem{definition}[theorem]{Definition}
\newtheorem{example}[theorem]{Example}

\theoremstyle{remark}
\newtheorem{remark}[theorem]{Remark}
\numberwithin{equation}{section}

\newcommand{\nb}{\mathrm{T}}

\newcommand{\Ac}{\mathrm{Ac}}
\newcommand{\CoAc}{\mathrm{Coac}}
\newcommand{\Trim}{\mathrm{Trim}}

\begin{document}

\begin{abstract}                          
The paradigm of Cyber--Physical Systems of Systems (CPSoS) is becoming rather popular in the control systems research community because of its expressive power able to properly handle many engineered complex systems of interest. Decentralized control techniques offer a promising approach in taming the inherent complexity of CPSoS, also connected with the design of needed communication infrastructures and computing units. In this paper, we propose decentralized control of networks of discrete--time nonlinear control systems, enforcing complex specifications expressed in terms of regular languages, within any desired accuracy. As discussed in the paper, regular languages, while being traditionally studied in the research community of discrete--event systems, also provide a useful mean to model a rather wide variety of complex specifications for control systems. The design of decentralized controllers is based on formal methods and in particular, on the use of discrete abstractions. Efficient synthesis of such controllers is derived by resorting to on--the--fly algorithmic techniques that also allow the use of parallel computing architectures. Advantages and disadvantages of the decentralized approach over a centralized one, also in terms of computational complexity, are discussed. An illustrative example is presented, which shows the applicability and effectiveness of the results proposed. 
\end{abstract}

\title[Decentralized Supervisory Control of Networks of Nonlinear Control Systems]{Decentralized Supervisory Control of \\Networks of Nonlinear Control Systems}

\thanks{The research leading to these results has been partially supported by the Center of Excellence DEWS.}

\author[Giordano Pola, Pierdomenico Pepe and Maria D. Di Benedetto]{
Giordano Pola$^{1}$, Pierdomenico Pepe$^{1}$ and Maria D. Di Benedetto$^{1}$}
\address{$^{1}$
Department of Information Engineering, Computer Science and Mathematics, Center of Excellence DEWS,
University of L{'}Aquila, 67100 L{'}Aquila, Italy}
\email{ \{giordano.pola,pierdomenico.pepe,mariadomenica.dibenedetto\}@univaq.it}

\maketitle

\section{Introduction}\label{sec1}
The novel paradigm of Cyber--Physical Systems of Systems (CPSoS) offers a solid framework where to model, analyze and design complex systems arising in diverse application domains of interest, as for example, smart transportation and logistics, smart power systems and smart buildings, efficient industrial production and gas, water etc. networks. Following \cite{CPSoS}, CPSoS are characterized by large, often spatially distributed physical systems with complex dynamics, distributed control, supervision and management, partial autonomy of the subsystems, dynamic reconfiguration of the overall system on different time--scales, possibility of emerging behaviors, continuous evolution of the overall system during its operation. \\
Decentralized control techniques offer a promising approach in taming the inherent complexity of CPSoS. Main advantages of decentralized control architectures over the centralized ones are: (i) they are effective in cases where full state of possibly spatially distributed plants cannot be accessed by a centralized controller, because of potential physical constraints; (ii) they require no communication infrastructures and limited computing units resources with respect to centralized architectures; (iii) scalability, decentralized control architectures are often suited to control large--scale and distributed plants. \\
Several decentralized control techniques have been proposed in diverse research areas ranging from e.g. decentralized stabilization and regulation \cite{Davison76,SandellTAC78}, robust stabilization, optimization, reliability design \cite{Siljak12,RotkowitzTAC06,LinTAC13}, decentralized adaptive control \cite{Perutka00}, consensus and formation control problems in multi--agent systems \cite{EgerstedtBook, OlfatiPIEEE07,OlfatiTAC06}, with e.g. application to mobile robotics \cite{JadbabaieTAC03,Bullo09}, to decentralized supervisory control of discrete--event systems (DES), e.g. \cite{Rudie92,Cassandras}. In particular, the last research area offers a systematic approach to enforce complex specifications expressed in terms of regular languages on large--scale \textit{qualitative} systems as DES are. \\
The aim of this paper is to transfer this decentralized control design methodology from \textit{qualitative} systems to \textit{quantitative} systems, described by a network of discrete--time nonlinear (infinite states) control systems $\Sigma_i$. Decentralized control architecture consists of a collection of local controllers $C_i$, each one associated with $\Sigma_i$. 
As also generally assumed in decentralized control of dynamical systems, local controllers $C_i$ are not allowed to communicate. 
We focus on specifications expressed as regular languages, traditionally considered in the control design of DES. This class, when used in the control design of purely continuous (or hydrid) systems, is rather rich and, as also pointed out in \cite{TabuadaTAC08}, comprises reachability and motion planning specifications, periodic orbits, state-based switching specifications, specifications involving sequences of smaller tasks that need to be performed according to a given order. 
Moreover, operators known for regular languages, and for automata recognizing them, as for example concatenation, union, intersection and complement, see e.g. \cite{Cassandras}, provide a useful mean assisting the designer in properly modeling desired complex specifications. 
The approach that we use to solve our decentralized control problem is based on formal methods, see e.g. \cite{ModelChecking}, and in particular, on the use of discrete abstractions, also called symbolic models, see e.g. \cite{DiscAbs,paulo}. Symbolic models are abstract and \textit{finite}  descriptions of \textit{infinite states} control systems where each state corresponds to an aggregate of continuous states and each label to an aggregate of control inputs. 
In this regard, the recent work \cite{PolaTAC16}, proposing networks of symbolic models approximating networks of discrete--time nonlinear control systems, provides a useful framework where to solve our decentralized control problem. Indeed, once networks of symbolic models have been constructed, one can design decentralized controllers for purely symbolic/discrete processes, as e.g. DES are. 
However, current methodologies known for decentralized control of DES cannot be used in our framework because, as also recalled in the paper  through an example, local controllers in the DES domain contribute \textit{sequentially} in enforcing the global specification, whereas in our framework local controllers contribute \textit{concurrently} in enforcing the global specification. 
This key difference asks for classical results available for DES to be revisited. 
This is the object of investigation of the present paper. 
We design local controllers enforcing a given regular language specification on the original network of control systems, within any desired accuracy.  
By following the general ideas of on--the--fly algorithms as in e.g. \cite{onthefly2,onthefly3}, and in particular, by extending \cite{PolaTAC12} to a decentralized setting, we propose efficient controllers synthesis that can be also implemented via parallel computing architectures. 
A comparison with a centralized approach is formally discussed, which shows that the parts of the specification that can be enforced through a centralized control architecture and through a decentralized control architecture coincide; the only limitation of decentralized architectures 
lies in the need for local controllers to agree in advance on which part of the specification to enforce, which however,  
as discussed in the paper also through an example, is intrinsic of any decentralized control architecture not allowing local controllers to communicate. 
This is important because, even when physical constraints allow designing communication and computing infrastructures needed in centralized control architectures, one can use decentralized control architectures, thus saving resources for designing and effectively implementing needed infrastructures. 
Advantages in terms of computational complexity of proposed decentralized controllers over to the centralized ones are discussed. 
An illustrative example is also included, which shows the applicability and effectiveness of the results proposed. \\
To the best of our knowledge, formal methods techniques developed in this paper have not yet been explored in the current literature on decentralized control of dynamical systems, with the only exception of \cite{BorriNecSys2013} which however, is tailored on special classes of regular language specifications. \\
Preliminary results of this paper are reported in \cite{PolaCDC16dec}. This paper extends \cite{PolaCDC16dec} by providing novel results on on--the--fly algorithms, a formal comparison with centralized control architectures and an illustrative example. \\
The paper is organized as follows. Section \ref{sec2} introduces the decentralized control problem set--up. Section \ref{sec3} introduces 
some preliminary results that are used in Section \ref{sec4} to derive the solution to the control problem.
Formal comparison with centralized control architectures is discussed in Section \ref{sec5}. Section \ref{sec6} presents on--the--fly algorithms and derives computational complexity analysis. Section \ref{sec7} offers an illustrative example. Concluding remarks are given in Section \ref{sec8}. Notation and basic definitions are reported in the Appendix.

\section{Networks of control systems and problem formulation} \label{sec2}
The network of control systems that we consider in this paper is given as the interconnection of the discrete--time nonlinear control systems $\Sigma_1,\Sigma_2,...,\Sigma_N$ described by:
\begin{equation}
\label{NetCS}
\Sigma_i :
\left\{
\begin{array}
{l}
x_{i}(t+1)=f_i(x_1(t),
...,x_N(t),u_i(t)),\\
x_{i}(t) \in \mathbb{R}^{n_i},u_{i}(t) \in \mathcal{U}_i \subset \mathbb{R}^{m_i},t \in \mathbb{N}_{0}.
\end{array}
\right.
\end{equation}

\noindent Let $n=\sum_{i\in [1;N]}n_i$ and $m=\sum_{i\in [1;N]}m_i$. 
Functions $f_{i}:\mathbb{R}^{n} \times \mathbb{R}^{m_i} \rightarrow \mathbb{R}^{n_i}$ are assumed to be continuous in their arguments and satisfying $f_i(0_{n},0_{m_i})=0_{n_i}$. 
Sets $\mathcal{U}_{i}$ are assumed to be finite and containing the origin $0_{m_i}$; this assumption is motivated by concrete applications in e.g. CPSoS where control inputs can only assume a finite number of values. 
A trajectory of $\Sigma_i$ is a function 
\[
x_i:[0;t_{f,i}] \rightarrow \mathbb{R}^{n_i}
\]
satisfying (\ref{NetCS}) for all times $t\in [0;t_{f,i}[$, for some time $t_{f,i}\in\mathbb{N}_0$, where we set $[0;0[=\varnothing$. 
For coincise notation, we may also refer to the network of control systems in (\ref{NetCS}) by the control system 
\begin{equation}
\label{sigma}
\Sigma:
\left\{
\begin{array}
{l}
x(t+1)=f(x(t),u(t)),\\
x(t) \in \mathbb{R}^{n},u(t) \in \mathcal{U} \subset \mathbb{R}^{m}, t \in \mathbb{N}_{0},
\end{array}
\right.
\end{equation}
where $\mathcal{U}=\bigtimes_{i\in [1;N]} \mathcal{U}_{i}$ and $f(x,(u_1,...,u_N))=(f_1(x,u_1),$ $...,f_N(x,u_N))$ for any $x\in \mathbb{R}^{n}$ and $(u_1,...,u_N)\in \mathbb{R}^{m}$. 
A trajectory of $\Sigma$ is a function 
\begin{equation}
\label{traj}
x:[0;t_f] \rightarrow \mathbb{R}^{n}
\end{equation}
satisfying (\ref{sigma}) for all times $t\in [0;t_f[$, for some time $t_f\in\mathbb{N}_0$.
Trajectory 
\begin{equation}
\label{trajbis}
x':[0;t'_f] \rightarrow \mathbb{R}^{n}
\end{equation}
is said to be a continuation of trajectory $x(\cdot)$ as in (\ref{traj}), if $t_f < t'_f$ and $x'(t)=x(t)$ for all $t\in [0;t_f]$. \\
We now formalize the class of specifications we focus on in this paper. Let $Y_{Q}$ be a finite subset of the state space $\mathbb{R}^{n}$ of $\Sigma$. The specification is expressed as a regular language 
\begin{equation}
\label{spec}
L_Q\subset Y_{Q}^\ast,
\end{equation}
where $Y_{Q}^\ast$ is the Kleene closure of $Y_{Q}$. This class of specifications is rather rich and comprises, as also pointed out in \cite{TabuadaTAC08}, reachability and motion planning specifications, periodic orbits, state-based switching specifications, specifications involving sequences of smaller tasks that need to be performed according to a given order. \\ 
We now define the class of decentralized controllers we consider. To this purpose, consider the directed graph $\mathcal{G}=(\mathcal{V},\mathcal{E})$, describing the interaction among subsystems $\Sigma_i$ in the network, where $\mathcal{V}=[1;N]$ and $(j,i)\in \mathcal{E}$, if function $f_{i}$ of $\Sigma_{i}$ depends explicitly on variable $x_{j}$ or equivalently, there exist $y_{j},z_{j}\in \mathbb{R}^{n_j}$ such that $f_{i}(x_1,...,x_{j-1},y_{j},$ $x_{j+1},...,x_n,u_{i}) \neq f_{i}(x_1,...,x_{j-1},z_{j},x_{j+1},...,x_n,u_{i})$ for some $x_k\in \mathbb{R}^{n_k}$, $k\in [1;N]$ and $k\neq j$, and $u_i\in\mathcal{U}_i$. Moreover, for any $i \in [1;N]$ define
\begin{equation}
\label{N(i)}
\mathcal{N}(i)=\{j \in \mathcal{V} : (j ,i) \in \mathcal{E}\}.
\end{equation}
For later purposes, for any $i\in [1;N]$ consider also the functions 
\begin{equation}
\label{psi}
\psi_i : \mathbb{R}^{n_i} \times (\bigtimes_{j\in\mathcal{N}(i)} \mathbb{R}^{n_j} ) \times \mathcal{U}_i\rightarrow \mathbb{R}^{n_i}
\end{equation}
such that $\psi_i(x_i,w_i,u_i)=f_i(x,u_i)$ with $w_i=(x_{j_1},x_{j_2},...,x_{j_{l_i}})\in (\bigtimes_{j_s\in\mathcal{N}(i)} \mathbb{R}^{n_{j_s}})$, for all $u_i\in\mathcal{U}_i$ and all $x=(x_1,x_2,...,x_N)\in\mathbb{R}^n$. 
We assume a decentralized architecture for the controller which is then specified as a collection of local dynamic controllers $C_i$, where $C_i$ is associated with $\Sigma_i$, in the form of
\begin{equation}
\label{Ci}
C_i:
\left\{
\begin{array}
{l}
x_{c,i}(t+1) = f_{c,i}(x_{c,i}(t)),\\
u_{i}(t) \in h_{c,i}(x_{c,i}(t))\subseteq \mathcal{U}_i,\\
x_{c,i}(0) \in X_{c,i}^0,\\
x_{c,i}(t) \in X_{c,i},t\in\mathbb{N}_0,
\end{array}
\right. 
\end{equation}
where $x_{c,i}(t)$ is the state of $C_i$ and $u_{i}(t)$ is the output of $C_i$ at time $t$. 
Controllers $C_i$ are open--loop, i.e. they do not depend on the current states $x_i(t)$ and $x_j(t)$ with $j\in \mathcal{N}(i)$, as instead often assumed in decentralized control of dynamical systems. We defer to Remarks \ref{remstatic} and \ref{remTSC} a discussion in this regard. 
Let 
\[
C=(C_1,C_2,...,C_N)
\]
be the decentralized controller applied to the network of control systems $\Sigma_i$. Interaction between control systems $\Sigma_i$ in the network and local controllers $C_i$ is obtained by coupling Eqns. (\ref{NetCS}) and (\ref{Ci}), for all $i\in [1;N]$, and denoted as $\Sigma^C$. Control system $\Sigma^C$ may exhibit blocking behaviors. In fact, existence of a trajectory $x(\cdot)$ of $\Sigma^C$ in the form of (\ref{traj}) implicitly requires that for all times $t\in [0;t_f[$ and $i\in [1;N]$:
\begin{equation}
\label{Cnonblock}
h_{c,i}(x_{c,i}(t))\neq\varnothing.
\end{equation}
We can now formalize the control problem we consider: 

\begin{problem}
\label{problem}
Given the network of control systems $\Sigma_i$ in (\ref{NetCS}), the regular language specification $L_Q$ in (\ref{spec}) and a desired accuracy $\theta\in\mathbb{R}^+$, find 
a set of initial states $\mathcal{X}_0 \subseteq \mathbb{R}^n$, a set of final states $\mathcal{X}_f \subseteq \mathbb{R}^n$ and 
a collection of local controllers $C_i$ in (\ref{Ci}) 
such that:
\begin{itemize}
\item[(i)] for any trajectory $x(\cdot)$ of $\Sigma^C$ as in (\ref{traj}) with $x(0)\in \mathcal{X}_0$, either $x(t_f)\in \mathcal{X}_f$ or there exists a continuation $x'(\cdot)$ of $x(\cdot)$, as in (\ref{trajbis}), such that $x'(t'_f)\in \mathcal{X}_f$;
\item[(ii)] for any trajectory $x(\cdot)$ of $\Sigma^C$ as in (\ref{traj}) with $x(0)\in \mathcal{X}_0$ and $x(t_f)\in \mathcal{X}_f$, there exists a word $q_0 q_1 ... q_{t_f} \in L_Q$ such that for all times $t \in [0;t_f]$:
\begin{equation}
\label{condP2}
\Vert x(t) - q_t \Vert \leq\theta.
\end{equation}
\end{itemize}
\end{problem}

Since condition (\ref{condP2}) relaxes condition $x(t)=q_t$, Problem \ref{problem} can be thought of as an approximate version of classical decentralized supervisory control problems traditionally given for DES (see e.g. \cite{Cassandras}) and here extended to networks of (infinite states) nonlinear control systems. Further discussion in this regard is reported in Remark \ref{remTSC} in Section \ref{sec5}.\\
For later purposes, we give the following
\begin{definition}
\label{DefPartDec}
Triplet $(C,\mathcal{X}_0,\mathcal{X}_f)$ is said to enforce a word $q_0 q_1 ... q_{t_f} \in L_Q$ within accuracy $\theta$ on $\Sigma$, if 
there exists a trajectory $x(\cdot)$ of $\Sigma^C$ as in (\ref{traj}) with $x(0)\in \mathcal{X}_0$ and $x(t_f)\in \mathcal{X}_f$ satisfying (\ref{condP2}) for all times $t \in [0;t_f]$.
\end{definition}

\section{Approximating networks of control systems} \label{sec3}
In this section we propose some results based on \cite{PolaTAC16} and concerning the construction of networks of symbolic models approximating networks of control systems. 
A symbolic model is an abstract description of a control system where each state corresponds to an aggregate of continuous states and each label to an aggregate of control inputs. 
We start by giving a representation of $\Sigma$ in terms of systems:
\begin{definition}
\label{syssigma}
Given $\Sigma$, define the system 
\[
S(\Sigma)=(X,X_0,U,\rTo,X_m,Y,H)
\]
where $X=X_0=X_m=\mathbb{R}^n$, $U=\mathcal{U}$, $x \rTo^{u} x^+$ if $x^{+}=f(x,u)$, $Y=\mathbb{R}^n$, and $H(x)=x$ for any $x\in \mathbb{R}^n$. 
\end{definition}

System $S(\Sigma)$ is metric when we regard $Y=\mathbb{R}^n$ as being equipped with the metric 
$\mathbf{d}(x,x')=\max_{i\in [1;N]} 
\mathbf{d}_i(x_i,x'_i)$, for all $x=(x_1,x_2,...,x_N),x'=(x'_1,x'_2,...,x'_N) \in \mathbb{R}^n$, 
where $\mathbf{d}_i$, defined by $\mathbf{d}_i(x_i,x'_i)=\Vert x_i-x'_i\Vert$ for all $x_i,x'_i\in\mathbb{R}^{n_i}$, is the metric used for $\mathbb{R}^{n_i}$. System $S(\Sigma)$ will be approximated by means of networks of systems that are introduced in the following

\begin{definition} 
\label{symbmod}
Given $\Sigma_i$, $i\in [1;N]$ and a quantization vector $\eta \in\mathbb{R}^{+}_N$, define the system 
\[
S^{\eta}(\Sigma_i)=(X^{\eta}_{i},X^{\eta}_{i,0},W^{\eta}_i\times U^{\eta}_i,\rTo_{\eta,i},X^{\eta}_{i,m},Y^{\eta}_i,H^{\eta}_i)
\]
where:
\begin{itemize}
\item $X^{\eta}_i=X^{\eta}_{i,0}=X^{\eta}_{i,m}=\eta(i) \mathbb{Z}^{n_i}$;
\item $W^{\eta}_i=\bigtimes_{j\in\mathcal{N}(i)} \eta(j)\mathbb{Z}^{n_j}$;
\item $U^{\eta}_i=\mathcal{U}_i$;
\item $\xi_i \rTo^{(w_i,u_i)}_{\eta,i} \xi^{+}_i$, if $\xi^{+}_i=[\psi_i(\xi_i,w_i,u_i)]_{\eta(i)}$ 
with $\psi_i$ in (\ref{psi});
\item $Y^{\eta}_i=\mathbb{R}^{n_i}$;
\item $H^{\eta}_i(\xi_i)=\xi_i$ for any $\xi_i\in X^{\eta}_i$.
\end{itemize}
\end{definition}

Each system $S^{\eta}(\Sigma_i)$ approximates each control system $\Sigma_i$ in the network for any desired accuracy. It is countable and becomes symbolic when one is interested in the dynamics of $\Sigma_i$ on a bounded subset of $\mathbb{R}^{n_i}$, as in most applications of interest and also in this paper, see Section \ref{sec4}. 
System $S^{\eta}(\Sigma_i)$ is metric with metric $\mathbf{d}_i$. 
By definition of the transition relation $\rTo_{\eta,i}$ and since operator $[\cdot]_{\eta(i)}$ is a function, system $S^{\eta}(\Sigma_i)$ is deterministic. 
Interaction among systems $S^{\eta}(\Sigma_i)$ is formalized by the following 
\begin{definition}
\label{lab}
\cite{PolaTAC16}
Given $S^{\eta}(\Sigma_i)$, $i\in [1;N]$, define the network of systems 
\[
\mathcal{S}(\{S^{\eta}(\Sigma_i)\}_{i\in [1;N]})= (X^{\eta},X^{\eta}_0,U^{\eta},\rTo_{\eta},X^{\eta}_m,Y^{\eta},H^{\eta})
\]
where:
\begin{itemize}
\item $X^{\eta}=X^{\eta}_0=X^{\eta}_m=\bigtimes_{i\in [1;N]} X^{\eta}_{i}$;
\item $U^{\eta}=\bigtimes_{i\in [1;N]}U^{\eta}_{i}$;
\item $(\xi_{1}, ...,\xi_{N})\rTo_\eta^{(u_1, ...,u_N)} (\xi_{1}^{+}, ...,\xi_{N}^{+})$, if $\xi_{i} \rTo_{\eta,i}^{(w_i,u_i)} \xi_{i}^{+}$ with $w_i=(\xi_{j_1},\xi_{j_2},...,\xi_{j_{l_i}})$, $j_s\in\mathcal{N}(i)$ for any $i\in [1;N]$;
\item $Y^{\eta}=\bigtimes_{i\in [1;N]}Y^{\eta}_{i}$;
\item $H^{\eta}(\xi_1, ...,\xi_N)=(H^{\eta}_1(\xi_1), ...,H^{\eta}_N(\xi_N))$. 
\end{itemize}
\end{definition}

System $\mathcal{S}(\{S^{\eta}(\Sigma_i)\}_{i\in [1;N]})$ is metric with metric $\mathbf{d}$ and inherits from systems $S^{\eta}(\Sigma_i)$ the properties of being deterministic and countable/symbolic. 
In the sequel we consider the following 
\begin{assumption}
\label{A1}
There exists a locally Lipschitz, incrementally globally asymptotically stable ($\delta$--GAS) Lyapunov function (see \cite{IncrementalS}) 
\begin{equation}
\label{Lfunction}
V: \mathbb{R}^{n} \times \mathbb{R}^{n} \rightarrow \mathbb{R}^{+}_{0}
\end{equation}
for $\Sigma$, i.e. function $V$ satisfies the following inequalities for all $x,x'\in\mathbb{R}^{n}$ and $u\in\mathcal{U}$:
\begin{itemize}
\item[(i)] $\underline{\alpha}(\left\Vert x-x' \right\Vert )\leq V(x,x')\leq \overline{\alpha}(\left\Vert x-x' \right\Vert )$, 
\item[(ii)] $V(f(x,u),f(x',u)) - V(x,x') \leq -\rho ( V(x,x') )$,
\end{itemize}
for some $\mathcal{K}_{\infty}$ functions $\underline{\alpha}$, $\overline{\alpha}$, $\rho$.
\end{assumption}
Throughout the paper we assume the existence of a $\mathcal{K}_{\infty}$ function $\sigma$ such that the $\delta$--GAS Lyapunov function $V$ satisfies for all $x,y,z\in\mathbb{R}^n$
\begin{equation}
\label{ineqnew}
\vert V(x,y) - V(x,z) \vert \leq \sigma (\Vert y-z \Vert ).
\end{equation}
\noindent
The above assumption is not restrictive since in order to solve Problem \ref{problem} we are interested in the dynamics of $\Sigma$ on a bounded subset of $\mathbb{R}^{n}$ (see Section \ref{sec4}). 
We now have all the ingredients to present the following

\begin{proposition}
\label{ThSCC}
Suppose that Assumption \ref{A1} holds. Then, for any desired accuracy $\mu\in\mathbb{R}^{+}$ and for any quantization vector $\eta \in \mathbb{R}^{+}_{N}$ satisfying the following inequality
\begin{equation}
\Vert \eta \Vert \leq \min \left\{
(\sigma^{-1} \circ \rho \circ \underline{\alpha})(\mu),(\overline{\alpha}^{-1} \circ \underline{\alpha})(\mu)
\right\},
\label{statem}
\end{equation}
relation $\mathcal{R}_{\mu}\subseteq \mathbb{R}^n \times X^{\eta}$ specified by
\begin{equation}
\label{relinit}
(x,\xi)\in \mathcal{R}_{\mu} \Leftrightarrow V(x,\xi) \leq \underline{\alpha}(\mu)
\end{equation}
is a strong $\mu$--approximate bisimulation between $S(\Sigma)$ and $\mathcal{S}(\{S^{\eta}(\Sigma_i)\}_{i\in [1;N]})$. Consequently, systems $S(\Sigma)$ and $\mathcal{S}(\{S^{\eta}(\Sigma_i)\}_{i\in [1;N]})$ are strongly $\mu$-bisimilar.
\end{proposition}

\smallskip

\begin{proof}
Direct consequence of Proposition 1 in \cite{PolaTAC16}.
\end{proof}

\medskip

The above result requires the existence of a $\delta$--GAS Lyapunov function for $\Sigma$. Compositional design of such Lyapunov function can be done by resorting e.g. to the small--gain theorem, see for instance \cite{SmallGainTh2}. These arguments have been used in \cite{PolaTAC16} to derive networks of symbolic models approximating networks of discrete--time nonlinear control systems. 
The main difference between the results reported in this section and in \cite{PolaTAC16} are: (i) while systems in Definition \ref{symbmod} are deterministic, those in \cite{PolaTAC16} are nondeterministic; (ii) quantization parameters $\eta(i)$ can be selected here independently from $\eta(j)$ but in the respect of (\ref{statem}), while selection of $\eta(i)$ depends on the selection of some other $\eta(j)$ in \cite{PolaTAC16}; (iii) sets $\mathcal{U}_i$ are finite here while they are convex, bounded and with interior in \cite{PolaTAC16}.

\section{Decentralized supervisory control design} \label{sec4}

In this section we provide the solution to Problem \ref{problem}. By using the results in Section \ref{sec3}, the design of decentralized controllers can be translated from a continuous (infinite states) domain to a symbolic (finite) domain. Hence, one could in principle use techniques available for DES to design decentralized controllers, see e.g. \cite{Rudie92,Cassandras}. 
However, these techniques cannot be used in our framework because while local controllers in Problem \ref{problem} contribute \textit{concurrently} in enforcing the global specification, in decentralized supervisory control of DES, local controllers contribute \textit{sequentially} in enforcing the global specification, as briefly recalled in the following example.

\begin{example}
\label{example} 
Consider a finite system $S$ with $\mathcal{L}_m^u(S)=\mathcal{U}^\ast$ where $\mathcal{U}^\ast$ denotes the Kleene closure of $\mathcal{U}=\{a,b,c\}$. Consider the regular language specification $L_Q=\{\varepsilon,a,ab,aba,abab,ababa,...\}$ where $\varepsilon$ is the empty word. Suppose that supervisors (controllers) $C_1$ and $C_2$ are characterized by sets of controllable\footnote{We refer to e.g. \cite{Cassandras} for the notions of controllable or observable events.} events $\mathcal{U}_1=\{a,c\}$ and $\mathcal{U}_2=\{b,c\}$, respectively. Assume further that the set of observable$^1$ events of $C_1$ and $C_2$ coincide in $\mathcal{U}$. A decentralized control policy enforcing $L_Q$ on $S$ is as follows: supervisor $C_1$ enforces event $a$ after having measured events $b$ and $\varepsilon$; supervisor $C_2$ enforces event $b$ after having measured event $a$. 
\end{example}
\smallskip
Motivated by inherent differences between decentralized control schemes used for DES and in Problem \ref{problem}, we now extend techniques of decentralized supervisory control from DES to our problem set--up. 
We start with the following 

\begin{example}
\label{examplenew}
Consider a network of two control systems $\Sigma_i$ described by $x_i(t+1)=-2x_i(t)+u_i(t)$, $t\in\mathbb{N}_0$ with $x_i(t) \in \mathbb{R}$, $\mathcal{U}_i=[-1;1]$ and a specification $L_Q$ described by the collection of words $(0,0)(1,1)$ and $(0,0)(-1,-1)$; set for simplicity the desired accuracy to $\theta=0$. First of all, corresponding control system $\Sigma$ satisfies Assumption \ref{A1}. In order for the specification to be enforced by a decentralized controller $C=(C_1,C_2)$, the controllers $C_i$ need to agree on which part of the specification they want to enforce. Indeed, if they want to enforce word $(0,0)(1,1)$, they both select at time $t=0$ control input $u_i(0)=1$; instead, if they want to enforce word $(0,0)(-1,-1)$, they both select at time $t=0$ control input $u_i(0)=-1$. If the controllers $C_i$ do not agree on which word of the specification $L_Q$ to enforce, $L_Q$ cannot be met by using \textit{any} decentralized control architecture. As a matter of fact, if $u_1(0)=1$ with the purpose of enforcing word $(0,0)(1,1)$, and if $u_2(0)=-1$ with the purpose of enforcing word $(0,0)(-1,-1)$, the state reached at time $t=1$ is $(1,-1)$ from which, $L_Q$ is not fulfilled. 
\end{example}
\begin{remark}
\label{remarknew}
At a general level, the problem raised in the above example can be solved as follows:\\
\textit{(i) Restriction of the class of specifications.} It is easy to see that the above problem is solved when $L_Q$ is "decoupled", i.e. it can be expressed as $L_Q=L_{Q,1} \times L_{Q,2} \times ... \times L_{Q,N}$ where each $L_{Q,i}$ is a regular language taking values in the projection of $Y_Q^\ast$ onto $\mathbb{R}^{n_i}$ and each $L_{Q,i}$ is enforced by a local controller $C_i$ which can be designed independently from any other $C_j$. This is for instance, the approach taken in \cite{BorriNecSys2013}. \\
\textit{(ii) Online agreement on the specification word to enforce.} 
When local controllers are allowed to share information through a to--be designed and implemented communication infrastructure, thus leading to a distributed control architecture, above problem can be solved because controllers can agree online on which word of the specification to enforce. \\ 
\textit{(iii) Offline agreement on the specification word to enforce.} 
When local controllers are not allowed to communicate, as in the decentralized control architecture we consider, 
controllers can only agree offline and hence in advance on which word of the specification to enforce. The advantage of this approach over the first one is that no restriction on the class of specifications is needed and, over the second approach, is that no communication infrastructure is required. In this paper we follow the third approach.
\end{remark}

We suppose that assumption of Proposition \ref{ThSCC} hold and use strong approximate bisimulation relation $\mathcal{R}_{\mu}$ defined in (\ref{relinit}). Let the system
\[
S'_Q=(X'_Q,X'_{0,Q},Y_Q,\rTo_{\prime,Q},X'_{Q,m},Y'_Q,H'_Q)
\]
be symbolic, deterministic, accessible and nonblocking and such that its input marked language coincides with the language specification, i.e. 
$
\mathcal{L}_m^u (S'_Q)=L_Q
$. 
Automatic tools for constructing $S'_Q$ are well known in the literature, see e.g. \cite{Jflap}. Given $S'_Q$, it is useful to define symbolic system $S_Q$ whose states are transitions of $S'_Q$ and vice versa. More formally:

\begin{definition}
Given system $S'_Q$, define system
\begin{equation}
\label{specsys}
S_Q=(X_Q,X_{Q,0},U_Q,\rTo_{Q},X_{Q,m},\mathbb{R}^n,H_Q)
\end{equation} 
where:
\begin{itemize}
\item $X_Q=\rTo_{\prime,Q}$;
\item $X_{Q,0}$ is the collection of states $x'_Q \rTo_{\prime,Q}^{u'_Q} x^{\prime,+}_Q$ in $X_Q$ with $x'_Q \in X'_{Q,0}$;
\item $U_Q=\{u_Q\}$, where $u_Q$ is a dummy input;
\item $\rTo_{Q}$ is the collection of transitions 
\[
\left(x^1_Q \rTo_{\prime,Q}^{u'_Q} x^2_Q\right) \rTo_Q^{u_Q} \left(x^3_Q \rTo_{\prime,Q}^{u'_Q} x^4_Q\right)
\]
with $x^2_Q=x^3_Q$;
\item $X_{Q,m}$ is the collection of states $x'_Q \rTo_{\prime,Q}^{u'_Q} x^{\prime,+}_Q$ in $X_Q$ with $x^{\prime,+}_Q \in X'_{Q,m}$;
\item $H_Q(x'_Q \rTo_{\prime,Q}^{u'_Q} x^{\prime,+}_Q)=u'_Q$ for any state $x'_Q \rTo_{\prime,Q}^{u'_Q} x^{\prime,+}_Q$ in $X_Q$.
\end{itemize}
\end{definition}

From the above definitions it is readily seen that 
\[
\mathcal{L}^y(S_Q)=\mathcal{L}^u(S'_Q), \quad \mathcal{L}_m^y(S_Q)=\mathcal{L}_m^u(S'_Q)=L_Q.
\]
Moreover, $S_Q$ is symbolic, accessible and nonblocking.  
In the sequel for ease of notation we denote a state 
$x'_Q \rTo_{\prime,Q}^{u'_Q} x^{\prime,+}_Q$ 
of $X_Q$ by $x_Q$ and a transition $x_Q \rTo_Q^{u_Q} x_Q^+$ of $S_Q$ by $x_Q \rTo_Q x_Q^+$. 
For any $i\in [1;N]$, function 
\[
H_{Q,i}:X_Q \rightarrow \mathbb{R}^{n_i}
\]
denotes the "projection" of function $H_Q$ onto $\mathbb{R}^{n_i}$, i.e. for all $x_Q \in X_Q$, $H_{Q,i}(x_Q)=q^i$ if $H_Q(x_Q)=(q^1,q^2,...,q^N)$. 
Consider the operators:
\[
\begin{array}
{l}
\mathcal{I}_i: (\rTo_Q) \times \mathbb{R}^+_N\rightarrow \{\tt{True},\tt{False}\},\mathit{i\in[1;N]}, \\
\mathcal{I}: (\rTo_Q) \times \mathbb{R}^+_N\rightarrow \{\tt{True},\tt{False}\}.\notag
\end{array}
\]
Consider any $\eta\in\mathbb{R}^+_N$ and any transition $x_Q \rTo_Q x_Q^+$ in $S_Q$. Then, for all $i \in [1;N]$ set
\begin{equation}
\label{deccontr0}
\mathcal{I}_i (x_Q \rTo_Q x_Q^+,\eta) = \tt{True},\\
\end{equation}
if there exists a control input $u_i\in U^{\eta}_i$ of $S^{\eta}(\Sigma_i)$ such that 
\begin{equation}
\label{condCi0}
[H_{Q,i}(x_{Q})]_{\eta(i)} \rTo_{\eta,i}
^{(v_i,u_i)} {[H_{Q,i}(x^+_{Q})]_{\eta(i)}},\\
\end{equation}
where 
$v_i=([H_{Q,{j_1}}(x_{Q})]_{\eta(j_1)},...,[H_{Q,{j_{l_i}}}(x_{Q})]_{\eta(j_{l_i})})$,
with $j_s\in\mathcal{N}(i)$. 
If no $u_i\in U^{\eta}_i$ exists satisfying (\ref{condCi0}), set
\begin{equation}
\label{deccontr0false}
\mathcal{I}_i (x_Q \rTo_Q x_Q^+,\eta) = \tt{False}.
\end{equation}
Operator $\mathcal{I}_i$, when evaluated in $x_Q \rTo_Q x_Q^+$ and $\eta$, is then set to $\tt{True}$ if transition $x_Q \rTo_Q x_Q^+$ can be matched by the system $S^\eta(\Sigma_i)$ and $\tt{False}$, otherwise. 
Since conditions (\ref{condCi0}) involve set of transitions of $S_Q$ and set $U^\eta_i$ that are finite, operator $\mathcal{I}_i$ can be effectively computed in a finite number of steps. 
Define:
\begin{equation}
\label{I}
\mathcal{I}(x_Q \rTo_Q x_Q^+,\eta)=\bigwedge_{i \in [1;N]} \mathcal{I}_i(x_Q \rTo_Q x_Q^+,\eta).
\end{equation}
Define the subsystem 
\begin{equation}
\label{SQeta}
S_{Q,\eta}=(X_Q,X_Q^0,U_Q,\rTo_{Q,\eta},X_{Q,m},Y_Q,H_Q)
\end{equation}
of $S_Q$ as in (\ref{specsys}), where the transition relation $\rTo_{Q,\eta}\subseteq \rTo_{Q}$ contains all and only transitions $x_Q \rTo_Q x_Q^+$ of $S_Q$ satisfying the following condition: 
\begin{equation}
\label{condfusion}
\mathcal{I}(x_Q \rTo_Q x_Q^+,\eta)=\tt{True}.
\end{equation}

\noindent
System $S_{Q,\eta}$ captures all transitions of the specification system $S_Q$ that can be matched by the control system $\Sigma^C$. However, system $S_{Q,\eta}$ is blocking in general. Since the controllers in $\Sigma^C$ are required to fulfill condition (\ref{Cnonblock}), we need to extract from $S_{Q,\eta}$ a subsystem exhibiting nonblocking behavior. 
This is accomplished by computing the subsystem $\Trim(S_{Q,\eta})$ of $S_{Q,\eta}$ (see Appendix for the definition of $\Trim$) which is indeed nonblocking. \\
We can now provide the solution to Problem \ref{problem}. We will follow the third approach discussed in Remark \ref{remarknew}.
Consider any word $\mathbf{q}$ marked by $\Trim(S_{Q,\eta})$, i.e. such that $\mathbf{q} \in \mathcal{L}_m^y(\Trim(S_{Q,\eta}))$, and let
\begin{equation}
\label{Sq}
S_\mathbf{q}=(X_{\mathbf{q}},\{x_{\mathbf{q}}^0\},U_{\mathbf{q}},\rTo_{\mathbf{q}},\{x_{\mathbf{q},m}\},Y_{\mathbf{q}},H_{\mathbf{q}})
\end{equation}
be a symbolic, accessible and nonblocking system, marking $\mathbf{q}$, i.e. such that $\mathcal{L}_m^y(S_\mathbf{q})=\{\mathbf{q}\}$. System $S_\mathbf{q}$ is characterized by a unique successor of each state. 
For this reason, in the sequel we write any transition of $S_\mathbf{q}$ in the form of $x_{\mathbf{q}} \rTo_{\mathbf{q}} x_{\mathbf{q}}^+$ by omitting the corresponding label. 
For any $i\in [1;N]$, function 
\[
H_{\mathbf{q},i}:X_{\mathbf{q}} \rightarrow \mathbb{R}^{n_i}
\]
denotes the "projection" of function $H_{\mathbf{q}}$ onto $\mathbb{R}^{n_i}$, i.e. for all $x_{\mathbf{q}} \in X_{\mathbf{q}}$, $H_{\mathbf{q},i}(x_{\mathbf{q}})=q^i$ if $H_{\mathbf{q}}(x_{\mathbf{q}})=(q^1,q^2,...,q^N)$. \\
Define the following sets:
\begin{equation}
\label{init}
\begin{array}
{rcl}
\mathcal{X}_0 & = &\mathcal{R}_{\mu}^{-1} (\bigtimes_{i\in [1;N]} \{[H_{\mathbf{q},i}(x_{\mathbf{q}}^0)]_{\eta(i)}\}),\\
\mathcal{X}_f & = &\mathcal{R}_{\mu}^{-1} (\bigtimes_{i\in [1;N]} \{[H_{\mathbf{q},i}(x_{\mathbf{q},m})]_{\eta(i)}\}).
\end{array}
\end{equation}
Entities defining $C_i$ in (\ref{Ci}) are then specified by:  
\begin{equation}
\label{deccontr}
\begin{array}
{l}
X_{c,i}^0 = \{x_{\mathbf{q}}^0\} ,\\
X_{c,i}=X_{\mathbf{q}},\\
f_{c,i}(x_{\mathbf{q}})=x_{\mathbf{q}}^+, \text{ if } x_{\mathbf{q}} \rTo_{\mathbf{q}} x_{\mathbf{q}}^+,\\
h_{c,i}(x_{\mathbf{q}})=
\left\{
\begin{array}
{l}
u_i\in U_i^{\eta}| x_{\mathbf{q}}^+=f_{c,i}(x_{\mathbf{q}}) \text{ and } \\
{[H_{\mathbf{q},i}(x_{\mathbf{q}})]_{\eta(i)}} \rTo_{\eta,i}
^{(v_i,u_i)} {[H_{\mathbf{q},i}(x^+_{\mathbf{q}})]_{\eta(i)}} 
\end{array}
\right\},\\
\end{array}
\end{equation}
where
$v_i=([H_{\mathbf{q},{j_1}}(x_{\mathbf{q}})]_{\eta(j_1)},...,[H_{\mathbf{q},{j_{l_i}}}(x_{\mathbf{q}})]_{\eta(j_{l_i})})$, with 
$j_s\in\mathcal{N}(i)$.
\begin{remark}
Sets $X_{c,i}^0$, $X_{c,i}$ and function $f_{c,i}$ in (\ref{deccontr}) are the same for all $C_i$. 
This feature is essential to solve problems raised in Example \ref{examplenew} and discussed in Remark \ref{remarknew}. As a by--product, this choice has the advantage of requiring limited computational effort that is significant when the number $N$ of subsystems in the network becomes large. We also stress that computation of functions $f_{c,i}$ and $h_{c,i}$ can be done offline, see Section \ref{sec6}, which is important because it reduces online computational time needed by the controllers to ensure timely control action. 
\end{remark}

We now have all the ingredients to present the main result of this paper.

\begin{theorem}
\label{Thmaindec}
Suppose that Assumption \ref{A1} holds. For any desired accuracy $\theta\in\mathbb{R}^+$ select $\mu\in\mathbb{R}^+$ and $\eta\in \mathbb{R}^+_N$ satisfying (\ref{statem}) and 
\begin{equation}
\mu + \Vert \eta \Vert /2 \leq \theta, \label{condMain2}
\end{equation}
Then, sets $\mathcal{X}_0$ and $\mathcal{X}_f$ in (\ref{init}) and controllers $C_i$ in (\ref{Ci}) specified by (\ref{deccontr}) solve Problem \ref{problem}.
\end{theorem}

\begin{proof}
Since assumption of Proposition \ref{ThSCC} holds, by (\ref{statem}) we get $S(\Sigma) \cong_{\mu} \mathcal{S}(\{S^{\eta}(\Sigma_i)\}_{i\in [1;N]})$; we consider $\mathcal{R}_{\mu}$ in (\ref{relinit}) as strong $\mu$--approximate bisimulation relation between $S(\Sigma)$ and $\mathcal{S}(\{S^{\eta}(\Sigma_i)\}_{i\in [1;N]})$. 
Consider any trajectory $x(.)$ of $\Sigma^C$ as in (\ref{traj}) with initial condition $x(0)=(x_1 (0),x_2 (0),...,x_N (0))\in\mathcal{X}_0$. 
Pick $\xi(0)=(\xi_1(0),\xi_2(0),...,\xi_N(0))$ such that 
$\xi_i(0)=[H_{\mathbf{q},i}(x_{\mathbf{q}}^0)]_{\eta(i)}$, $i\in [1;N]$. 
By definition of $\mathcal{X}_0$ in (\ref{init}) we get:
\begin{equation}
(x(0),\xi(0)) \in \mathcal{R}_{\mu}.\label{condth1}
\end{equation}
\noindent Define $q_0 =H_{\mathbf{q}}(x^0_{\mathbf{q}})$. By (\ref{condth1}), definition of $\mathcal{R}_{\mu}$ and of $\xi(0)$ we get
\[
\begin{array}
{rcl}
\Vert x(0) - \xi(0) \Vert & \leq & \mu,\\
\Vert \xi(0) - q_0 \Vert &   =  & \Vert \xi(0) - H_{\mathbf{q}}(x^0_{\mathbf{q}}) \Vert\\
                         &   =  & \max_{i\in [1;N]} \Vert \xi_i(0) - H_{\mathbf{q},i}(x^0_{\mathbf{q}}) \Vert \\
												 & \leq & \max_{i\in [1;N]}  \eta(i)/2=\Vert \eta \Vert /2 																				
\end{array}
\]
which, combined with (\ref{condMain2}), yields
\begin{equation}
\label{diff0}
\begin{array}
{rcl}
\Vert x(0) - q_0 \Vert &   \leq  & \Vert x(0) - \xi(0) \Vert + \Vert \xi(0) -q_0 \Vert\\
                       &   \leq  & \mu + \Vert \eta \Vert /2 \leq \theta .												
\end{array}
\end{equation}
By the nonblocking property of $S_\mathbf{q}$, either (case $1$) $x^0_{\mathbf{q}}\in X_{\mathbf{q},m}$ or (case $2$) there exists a transition $x^0_{\mathbf{q}} \rTo_{\mathbf{q}} x^1_{\mathbf{q}}$. In case $1$, by (\ref{condth1}), definition of $\xi(0)$ and (\ref{init}), $x(0)\in\mathcal{X}_f$ and condition (i) of Problem \ref{problem} holds for $t=t_f=0$. Moreover, since $q_0 \in L_Q$ and by (\ref{diff0}), condition (ii) of Problem \ref{problem} holds as well for $t=t_f=0$. 
We now address case $2$. By definition of $S_\mathbf{q}$, transition $x^0_{\mathbf{q}} \rTo_{\mathbf{q}} x^1_{\mathbf{q}}$ satisfies condition (\ref{condfusion}). Hence, by (\ref{deccontr0}) for all $i\in [1;N]$ there exists $u_i(0)\in U^\eta_i$ satisfying 
\begin{equation}
\label{ggio}
[H_{\mathbf{q},i}(x_{\mathbf{q}}^0)]_{\eta(i)} \rTo_{\eta,i}^{(v_i(0),u_i(0))} {[H_{\mathbf{q},i}(x^1_{\mathbf{q}})]_{\eta(i)}}
\end{equation}
where 
$v_i(0)=([H_{\mathbf{q},j_1}(x_{\mathbf{q}}^0)]_{\eta(j_1)},...,[H_{\mathbf{q},j_{l_i}}(x_{\mathbf{q}}^0)]_{\eta(j_{l_i})})$,
for all $j_s\in\mathcal{N}(i)$. 
Hence, by definition of $h_{c,i}$ in (\ref{deccontr}), we get 
$
u_i(0) \in h_{c,i}(x_{\mathbf{q}}^0)\neq \varnothing$, 
for all $i\in [1;N]$,
from which, condition (\ref{Cnonblock}) holds for $t=0$. Let $u(0)=(u_1(0),$ $u_2(0),...,u_N(0))$ and $x(1)=f(x(0),u(0))$. By Definition \ref{syssigma} we get:
\begin{equation}
\label{x(1)}
x(0) \rTo^{u(0)} x(1).
\end{equation}
Pick $\xi(1)=(\xi_1(1),\xi_2(1),...,\xi_N(1))$ such that $\xi_i(1)=[H_{\mathbf{q},i}(x_{\mathbf{q}}^1)]_{\eta(i)}$, $i\in [1;N]$.
By (\ref{ggio}), definitions of $\xi_i(0)$ and $\xi_i(1)$ we 
get
\[
\xi_i(0) \rTo^{(v_i(0),u_i(0))}_{\eta,i} \xi_i(1),
\]
which implies, by definition of $u(0)$ and Definition \ref{lab}
\begin{equation}
\label{xi(1)}
\xi(0) \rTo^{u(0)}_{\eta} \xi(1),
\end{equation}
i.e. the above transition is in $\mathcal{S}(\{S^{\eta}(\Sigma_i)\}_{i\in [1;N]})$. 
By (\ref{x(1)}), (\ref{xi(1)}), determinism of $S^\eta(\Sigma)$, and definition of $\mathcal{R}_{\mu}$ we get:
\[
(x(1),\xi(1)) \in \mathcal{R}_{\mu}.
\]
\noindent We now use induction and show that if the following conditions (H1), (H2) and (H3) hold for some $\tau \in \mathbb{N}_0$\\
(H1) $(x(\tau),\xi(\tau)) \in \mathcal{R}_{\mu}$, where
\[
\begin{array}
{l}
\xi_i(\tau)=[H_{\mathbf{q},i}(x_{\mathbf{q}}^{\tau})]_{\eta(i)}, i\in [1;N],\\
\xi(\tau)=(\xi_1(\tau),\xi_2(\tau),...,\xi_N(\tau)),
\end{array}
\]
(H2) $x^0_{\mathbf{q}} \rTo_{\mathbf{q}} x^1_{\mathbf{q}}\rTo_{\mathbf{q}} {...} \rTo_{\mathbf{q}} x^{\tau}_{\mathbf{q}}$,\\
(H3) Condition (\ref{condP2}) is satisfied for all $t\in [0;\tau]$ where $q_t$ is defined by $q_t = H_{\mathbf{q}}(x^t_{\mathbf{q}})$, $t\in [0;\tau]$,\\
then one of the following conditions (T1) or (T2) hold: \\
(T1) $x(\tau)\in \mathcal{X}_f$ and the word $q_0 q_1 ... q_{\tau} \in L_Q$;\\
(T2) condition (\ref{Cnonblock}) holds for $t=\tau$ and for any $u_i(\tau)\in h_{c,i}(x^{\tau}_{\mathbf{q}})$, $i\in [1;N]$, by setting 
\begin{eqnarray}
&& \hspace{-5mm}
u(\tau)=(u_1(\tau),u_2(\tau),...,u_N(\tau)),\label{u(tau)}\\
&& \hspace{-5mm}
x(\tau+1)=f(x(\tau),u(\tau)),\label{x(tau+1)} \\
&& \hspace{-5mm} 
\xi_i(\tau+1)=[H_{\mathbf{q},i}(x_{\mathbf{q}}^{\tau+1})]_{\eta(i)}, i\in [1;N],\label{xii(tau+1)}\\
&& \hspace{-5mm}
\xi(\tau+1)=(\xi_1(\tau+1),\xi_2(\tau+1),...,\xi_N(\tau+1)),\label{xi(tau+1)}
\end{eqnarray}
the following conditions hold:\\
(T2.1) $(x(\tau+1),\xi(\tau+1)) \in \mathcal{R}_{\mu}$,\\
(T2.2) $x^0_{\mathbf{q}} \rTo_{\mathbf{q}} x^1_{\mathbf{q}}\rTo_{\mathbf{q}} {...} \rTo_{\mathbf{q}} x^{\tau}_{\mathbf{q}} \rTo_{\mathbf{q}} x^{\tau+1}_{\mathbf{q}}$,\\
(T2.3) Condition (\ref{condP2}) is satisfied for all $t\in [0;\tau+1]$ where $q_t$ is defined by $q_t = H_{\mathbf{q}}(x^t_{\mathbf{q}})$, $t\in [0;\tau+1]$.\\
Let us assume then that (H1)--(H3) hold. 
By the nonblocking property of $S_\mathbf{q}$, either (case $1$) $x^{\tau}_{\mathbf{q}}\in X_{\mathbf{q},m}$ or (case $2$) there exists a transition $x^{\tau}_{\mathbf{q}} \rTo_{\mathbf{q}} x^{\tau+1}_{\mathbf{q}}$. 
We start by addressing case $1$. By (H1) and the definition of $\mathcal{X}_f$ in (\ref{init}) we get $x(\tau)\in\mathcal{X}_f$. Hence, condition (i) of Problem \ref{problem} holds for $t_f=\tau$. Since $x^{\tau}_{\mathbf{q}}\in X_{\mathbf{q},m}$ then $q_0 q_1 ... q_{\tau} \in L_Q$ and (T1) is proven. Moreover by (H3), condition (ii) of Problem \ref{problem} holds for $t_f=\tau$. 
We now address case $2$. First of all (T2.2) holds. By definition of $S_\mathbf{q}$, transition $x^{\tau}_{\mathbf{q}} \rTo_{\mathbf{q}} x^{\tau+1}_{\mathbf{q}}$ satisfies condition (\ref{condfusion}). 
Hence, by (\ref{deccontr0}), for all $i\in[1;N]$ there exists $u_i(\tau)\in U^\eta_i$ satisfying 
\begin{equation}
\label{trans1i}
[H_{\mathbf{q},i}(x_{\mathbf{q}}^{\tau})]_{\eta(i)} \rTo_{\eta,i}^{(v_i(\tau),u_i(\tau))} {[H_{\mathbf{q},i}(x^{\tau+1}_{\mathbf{q}})]_{\eta(i)}}
\end{equation}
where 
$v_i(\tau)=([H_{\mathbf{q},j_1}(x_{\mathbf{q}}^{\tau})]_{\eta(j_1)},...,[H_{\mathbf{q},j_{l_i}}(x_{\mathbf{q}}^{\tau})]_{\eta(j_{l_i})})$,
for all $j_s\in\mathcal{N}(i)$.
Hence, by definition of $h_{c,i}$ in (\ref{deccontr}), 
$u_i(\tau) \in h_{c,i}(x_{\mathbf{q}}^{\tau})\neq \varnothing$, for all $i\in [1;N]$ from which,
condition (\ref{Cnonblock}) holds for $t=\tau$ as requested in (T2). 
By (\ref{x(tau+1)}) and Definition \ref{syssigma} we get:
\begin{equation}
\label{x(tau)}
x(\tau) \rTo^{u(\tau)} x(\tau+1).
\end{equation}
By (\ref{xii(tau+1)}) and (\ref{trans1i}) we get 
\[
\xi_i(\tau) \rTo^{(v_i(\tau),u_i(\tau))}_{\eta,i} \xi_i(\tau+1),i\in [1;N]
\]
which by (\ref{u(tau)}), (\ref{xi(tau+1)}) and Definition \ref{lab} implies
\begin{equation}
\label{xxi(tau+1)}
\xi(\tau) \rTo^{u(\tau)}_{\eta} \xi(\tau+1).
\end{equation}
By (\ref{x(tau)}), (\ref{xxi(tau+1)}), determinism of $S^\eta(\Sigma)$, and definition of $\mathcal{R}_{\mu}$, we get (T2.1). 
Moreover, set $q_{\tau+1}=H_{\mathbf{q}}(x^{\tau+1}_{\mathbf{q}})$. By (T2.1) and definition of $\xi(\tau+1)$ in (\ref{xi(tau+1)}) we get
\[
\begin{array}
{rcl}
\Vert x(\tau+1) - \xi(\tau+1) \Vert & \leq & \mu;\\
\Vert \xi(\tau+1) - q_{\tau+1} \Vert &   =  & \Vert \xi(\tau+1) - H_{\mathbf{q}}(x^{\tau+1}_{\mathbf{q}})) \Vert\\
                         &   =  & \max_{i\in [1;N]} \Vert \xi_i(\tau+1) - \\
												 &      & H_{\mathbf{q},i}(x^{\tau+1}_{\mathbf{q}})) \Vert \\
												 & \leq & \max_{i\in [1;N]}  \eta(i)/2=\Vert \eta \Vert /2 ,																					
\end{array}
\]
which, combined with (\ref{condMain2}) yields:
\begin{equation}
\label{diff00}
\begin{array}
{rcl}
\Vert x(\tau+1) - q_{\tau+1} \Vert &   \leq  & \Vert x(\tau+1) - \xi(\tau+1) \Vert + \\
                                   &         & \Vert \xi(\tau+1) -q_{\tau+1} \Vert\\
                       &   \leq  & \mu + \Vert \eta \Vert /2 \leq \theta .												
\end{array}
\end{equation}
The above inequality combined with (H3) implies (T2.3). Thus, (T2) is proven. \\
In order to conclude the proof we need to show that there exists a time $t_f\in\mathbb{N}$ such that $x(t_f)\in\mathcal{X}_f$. Since $S_\mathbf{q}$ is nonblocking there exists a time $t_f\in\mathbb{N}$ such that $x^{t_f}_{\mathbf{q}}\in X_{\mathbf{q},m}$ which implies by (H1) and (\ref{init}) that $x(t_f)\in\mathcal{X}_f$. 
\end{proof}

By Theorem \ref{Thmaindec}, definition of $\Trim(S_{Q,\eta})$ and Definition \ref{DefPartDec}, it is readily seen that: 

\begin{corollary}
\label{ThPartDec}
Suppose that Assumption \ref{A1} holds and select $\eta$ as required in Theorem \ref{Thmaindec}. Then, 
for any word $\mathbf{q} \in \mathcal{L}_m^y(\Trim(S_{Q,\eta}))$ there exists a triplet $(C,\mathcal{X}_0,\mathcal{X}_f)$ enforcing it on $\Sigma$, within accuracy $\theta$. 
\end{corollary}

By the proof of Theorem \ref{Thmaindec} (see the inequalities (\ref{diff0}) and (\ref{diff00})) the following result holds:
\begin{corollary}
\label{CoromaindecBis}
Suppose that Assumption \ref{A1} holds. For any desired accuracy $\theta\in\mathbb{R}^+$ select $\mu\in\mathbb{R}^+$ and $\eta\in \mathbb{R}^+_N$ satisfying (\ref{statem}) and
\begin{equation}
\label{statem3}
\mu\leq \theta.
\end{equation}
If $Y_Q\subset X^{\eta}$ then sets $\mathcal{X}_0$ and $\mathcal{X}_f$ in (\ref{init}) and controllers $C_i$ in (\ref{Ci}) specified by (\ref{deccontr}) solve Problem \ref{problem}.
\end{corollary}

We conclude this section by discussing the choice in the class of controllers $C_i$ in (\ref{Ci}).
\begin{remark}
\label{remstatic}
The class of local controllers $C_i$ in (\ref{Ci}) and specified by (\ref{deccontr}), shown in Theorem \ref{Thmaindec} to solve Problem \ref{problem}, comes out from the general class of specifications we consider and shares analogies with the theory of supervisory control, see also Remark \ref{remTSC}. 
When the word $\mathbf{q}=q_0 q_1 ... q_{t_f}$ used in (\ref{Sq}) to define system $S_\mathbf{q}$ satisfies the following property
\begin{equation}
\label{propq}
q_t=q_{t'} \Rightarrow q_{t+1}=q_{t'+1},\forall t,t'\in [0;t_f-1], t_f\geq 1,
\end{equation}
it is possible to show by a slight modification of the proof of Theorem \ref{Thmaindec} that dynamic and open--loop local controllers $C_i$ can be replaced by static local state feedback controllers $C'_i$ in the form of
\begin{equation}
\label{Cibis}
u_i \in C'_i(x_i,x_{j_1},...,x_{j_{l_i}}), j_s\in\mathcal{N}(i), i\in [1;N],
\end{equation}
as often assumed in decentralized control of dynamical systems, 
where the partial maps $C'_i:\mathbb{R}^{n_i} \times (\bigtimes_{j\in\mathcal{N}(i)} \mathbb{R}^{n_j} )\rightarrow 2^{\mathcal{U}_i}$ in (\ref{Cibis}) are 
specified for any state $x_{\mathbf{q}}$ of system $S_{\mathbf{q}}$ in (\ref{Sq}) by:
\[
\begin{array}
{l}
(x_1,...,x_N)\in \mathcal{R}_{\mu}^{-1}(([H_{\mathbf{q},{1}}(x_{\mathbf{q}})]_{\eta(1)},...,[H_{\mathbf{q},{N}}(x_{\mathbf{q}})]_{\eta(N)})),\\
{[H_{\mathbf{q},i}(x_{\mathbf{q}})]_{\eta(i)}} \rTo_{\eta,i}^{(v_i,u_i)} {[H_{\mathbf{q},i}(x^+_{\mathbf{q}})]_{\eta(i)}}$, $i\in [1;N],
\end{array}
\]
where $v_i=([H_{\mathbf{q},{j_1}}(x_{\mathbf{q}})]_{\eta(j_1)},...,[H_{\mathbf{q},{j_{l_i}}}(x_{\mathbf{q}})]_{\eta(j_{l_i})})
$, $j_s\in\mathcal{N}(i)$, $i\in [1;N]$. When instead, word $\mathbf{q}$ violates condition (\ref{propq}), the class of controllers $C'_i$ is not general enough for enforcing $\mathbf{q}$ because if $q_t=q_{t'}$ with $t \neq t'$ and $q_{t+1}\neq q_{t'+1}$, controllers $C'_i$ need to enforce transition from $q_t$ to $q_{t+1}$ at time $t$, and transition from $q_{t'}=q_{t}$ to $q_{t'+1}\neq q_{t+1}$ at time $t'\neq t$. 
%
\end{remark}

\section{Comparison with centralized control architectures} \label{sec5}
In this section we establish connections with centralized control architectures.   
A centralized controller for $\Sigma$ is specified by the dynamic open--loop controller:

\begin{equation}
\label{Cc}
C_c:
\left\{
\begin{array}
{l}
x_{c}(t+1) \in f_{c}(x_{c}(t)),\\
u(t) \in h_{c}(x_{c}(t))\subseteq \mathcal{U},\\
x_{c}(0) \in X_{c}^0,\\
x_{c}(t) \in X_{c},t\in\mathbb{N}_0,
\end{array}
\right. 
\end{equation}
where $x_{c}(t)$ is the state of $C_c$ and $u(t)$ is the output of $C_c$ at time $t$. While state evolution of $C_c$ in (\ref{Cc}) is nondeterministic, state evolution of $C_i$ in (\ref{Ci}) is deterministic. This is a consequence of the fact that local controllers $C_i$ need to agree in advance of which word of the specification to enforce.  
We denote by $\Sigma^{C_c}$ the control system obtained as coupling Eqns. (\ref{sigma}) and (\ref{Cc}). Problem \ref{problem} rewrites in a centralized setting as:

\begin{problem}
\label{problemc}
Given $\Sigma$ in (\ref{sigma}), $L_Q$ in (\ref{spec}) and $\theta\in\mathbb{R}^+$, find $\mathcal{X}_{0,c} \subseteq \mathbb{R}^n$, $\mathcal{X}_{f,c} \subseteq \mathbb{R}^n$ and $C_c$ in (\ref{Cc}) such that:
\begin{itemize}
\item[(i)] for any trajectory $x(\cdot)$ of $\Sigma^{C_c}$ as in (\ref{traj}) with $x(0)\in \mathcal{X}_{0,c}$, either $x(t_f)\in \mathcal{X}_{f,c}$ or there exists a continuation $x'(\cdot)$ of $x(\cdot)$, as in (\ref{trajbis}), such that $x'(t'_f)\in \mathcal{X}_{f,c}$;
\item[(ii)] for any trajectory $x(\cdot)$ of $\Sigma^{C_c}$ as in (\ref{traj}) with $x(0)\in \mathcal{X}_{0,c}$ and $x(t_f)\in \mathcal{X}_{f,c}$, there exists a word $q_0 q_1 ... q_{t_f} \in L_Q$ such that for all times $t \in [0;t_f]$ condition (\ref{condP2}) holds.
\end{itemize}
\end{problem}
For later purposes, we need the following
\begin{definition}
\label{DefPartCen}
Language $\mathbf{L}_{\mathbf{Q}}(\Sigma^{C_c})$ is the collection of all words $q_0 q_1 ... q_{t_f}\in L_Q$ for which there exists a 
trajectory $x(\cdot)$ of $\Sigma^{C_c}$ as in (\ref{traj}) with $x(0)\in \mathcal{X}_{0,c}$ and $x(t_f)\in \mathcal{X}_{f,c}$ satisfying (\ref{condP2}) for all times $t \in [0;t_f]$.
\end{definition}
By the above definition, $\mathbf{L}_{\mathbf{Q}}(\Sigma^{C_c})$ represents the part of $L_Q$ that can be enforced on $\Sigma$ by $C_c$.
The solution to Problem \ref{problemc} mimicks the one given for the decentralized case. Consider
\[
\mathcal{I}_c: (\rTo_Q) \times \mathbb{R}^+_N\rightarrow \{\tt{True},\tt{False}\}.
\]
For any transition $x_Q \rTo_Q x_Q^+$ of system $S_Q$ defined in the previous section
\begin{equation}
\label{condFusC}
\mathcal{I}_c (x_Q \rTo_Q x_Q^+,\eta)=\tt{True}
\end{equation}
if there exists $u=(u_1,u_2,...,u_N)\in U^\eta$ such that conditions (\ref{condCi0}) with $i\in [1;N]$ are jointly satisfied, and 
$\mathcal{I}_c (x_Q \rTo_Q x_Q^+,\eta)=\tt{False}$, otherwise. 
Define the subsystem 
\begin{equation}
\label{SQetac}
S^c_{Q,\eta}=(X^c_Q,X_Q^{0,c},U^c_Q,\rTo_{Q,\eta,c},X_{Q,m,c},Y^c_Q,H^c_Q)
\end{equation}
of $S_Q$, where $\rTo_{Q,\eta,c}\subseteq \rTo_{Q}$ contains all and only transitions $x_Q \rTo_Q x_Q^+$ of $S_Q$ satisfying (\ref{condFusC}). Define 
\begin{equation}
\label{trimcentr}
\Trim(S^c_{Q,\eta})=(X_{\nb},X_{\nb,0},U_{\nb},\rTo_{\nb},X_{\nb,m},Y_{\nb},H_{\nb}),
\end{equation}
and the following sets:
\begin{equation}
\label{initc}
\begin{array}
{rcl}
\mathcal{X}_0^c & = &\mathcal{R}_{\mu}^{-1} (\bigcup_{x_{\nb} \in X_{\nb,0}}
\bigtimes_{i\in [1;N]} \{[H_{\nb,i}(x_{\nb})]_{\eta(i)}\}),\\
\mathcal{X}_f^c & = &\mathcal{R}_{\mu}^{-1} (\bigcup_{x_{\nb} \in X_{\nb,m}}
\bigtimes_{i\in [1;N]} \{[H_{\nb,i}(x_{\nb})]_{\eta(i)}\}).
\end{array}
\end{equation}

\noindent For any $i\in [1;N]$, function 
\[
H_{\nb,i}:X_{\nb} \rightarrow \mathbb{R}^{n_i}
\]
denotes the "projection" of function $H_{\nb}$ onto $\mathbb{R}^{n_i}$, i.e. for all $x_{\nb} \in X_{\nb}$, $H_{\nb,i}(x_{\nb})=q^i$ if $H_{\nb}(x_{\nb})=(q^1,q^2,...,q^N)$. Entities defining controller $C_c$ in (\ref{Cc}) are then specified by:
\begin{equation}
\label{contrc}
\begin{array}
{l}
X_{c}^0 = X_{\nb,0} ,\\
X_{c}=X_{\nb},\\
f_{c}(x_{\nb})=\{x_{\nb}^+\in X_{\nb} |  x_{\nb} \rTo_{\nb} x_{\nb}^+\},\\
h_{c}(x_{\nb})=\\
\left\{
\begin{array}
{c}
u=(u_1,u_2,...,u_N)\in U^{\eta}|\exists x_{\nb}^+ \in f_{c}(x_{\nb}) \text{ s.t. }\\{[H_{\nb,i}(x_{\nb})]_{\eta(i)}} \rTo_{\eta,i}^{(v_i,u_i)} {[H_{\nb,i}(x^+_{\nb})]_{\eta(i)}}, i\in [1;N]
\end{array}
\right\},\\
\end{array}
\end{equation}
where 
$
v_i=([H_{\nb,j_1}(x_{\nb})]_{\eta(j_1)},...,[H_{\nb,j_{l_i}}(x_{\nb})]_{\eta(j_{l_i})})$ with
$j_s\in\mathcal{N}(i)$. 
The following result holds.

\begin{theorem}
\label{propcentr}
Suppose that Assumption \ref{A1} holds. For any desired accuracy $\theta\in\mathbb{R}^+$ select $\mu\in\mathbb{R}^+$ and $\eta\in \mathbb{R}^+_N$ satisfying (\ref{statem}) and (\ref{condMain2}). Then, sets $\mathcal{X}^c_0$ and $\mathcal{X}^c_f$ in (\ref{initc}) and controller $C_c$ in (\ref{Cc}) specified by (\ref{contrc}) solve Problem \ref{problemc}.
\end{theorem}

The proof of the above result follows the same reasoning as the proof of Theorem \ref{Thmaindec} and is therefore omitted. From the above result, it is readily seen that 

\begin{corollary}
\label{ThPartCen}
Suppose that Assumption \ref{A1} holds and select $\eta$ as required in Theorem \ref{propcentr}. Then, $\mathbf{L}_{\mathbf{Q}}(\Sigma^{C_c})=\mathcal{L}_m^y(\Trim(S^c_{Q,\eta}))$.
\end{corollary}

A direct consequence of Theorem \ref{propcentr} and Corollary \ref{ThPartCen} is the following

\begin{corollary}
\label{coro2}
Suppose that Assumption \ref{A1} holds and select $\eta$ as required in Theorem \ref{propcentr}. 
Then, there exists a controller $C^c$ as in (\ref{Cc}) such that 
\begin{equation}
\label{Ccmax0}
\mathbf{L}_{\mathbf{Q}}(\Sigma^{C_c})=L_Q 
\end{equation}
if and only if 
\begin{equation}
\label{Ccmax}
\Trim(S_Q)=\Trim(S^c_{Q,\eta}).
\end{equation}
Moreover, if condition (\ref{Ccmax}) holds, then $C^c$ in (\ref{Cc}) specified by (\ref{contrc}) is such that condition (\ref{Ccmax0}) holds.
\end{corollary}

\begin{proof}
By Corollary \ref{ThPartCen} and since $\Trim(S_Q)=S_Q$ we get
$\mathbf{L}_{\mathbf{Q}}(\Sigma^{C_c})  = \mathcal{L}^y_m(\Trim(S^c_{Q,\eta}))=\mathcal{L}^y_m(\Trim(S_Q)) = \mathcal{L}^y_m(S_Q)=\mathcal{L}^u_m(S'_Q)=L_Q$. 
The second part of the proof holds as a consequence of Theorem \ref{propcentr} and Corollary \ref{ThPartCen}.
\end{proof}

\begin{remark}
\label{remTSC}
Corollary \ref{coro2} states that a necessary and sufficient condition for the control system $\Sigma$ to implement the whole specification $L_Q$ up to a given accuracy $\theta$, is that the specification $L_Q$ is contained in the behavior of the control system $\Sigma$, up to the accuracy $\theta$. This result can be viewed as the counterpart in our setting, of the so--called nonblocking controllability theorem (NCT) in the theory of supervisory control of DES, see e.g. \cite{Cassandras}, establishing sufficient and necessary conditions for the existence of a controller enforcing a regular language specification and such that controlled plant is nonblocking. In particular, \textit{the controller solving the NCT is shown to be any DES marking the specification};  interaction between the plant and the controller is formalized through the notion of parallel composition, where \textit{the controller does not have information on the current state of the plant}. Analogies with the results reported above in this section are noticeable. Indeed, controller $C_c$ \textit{replicates} the part of \textit{the specification system} $S_Q$ which can be enforced by $\Sigma$ and \textit{is open--loop}.
\end{remark}

We conclude this section by establishing connections between the decentralized and centralized control architectures that we proposed. The following result holds.

\begin{theorem}
\label{thconnex}
Suppose that Assumption \ref{A1} holds and select $\eta$ as required in Theorem \ref{Thmaindec} (or equivalently, as required in Theorem \ref{propcentr}). Then, for any word $\mathbf{q} \in \mathbf{L}_{\mathbf{Q}}(\Sigma^{C_c})$ there exists a triplet $(C,\mathcal{X}_0,\mathcal{X}_f)$ enforcing it within accuracy $\theta$.
\end{theorem}

\begin{proof}
We start by showing $\rTo_{Q,\eta}=\rTo_{Q,\eta,c}$. 
Consider any transition 
$x_Q \rTo_{Q,\eta,c} x^+_Q$. By definition of $S^c_{Q,\eta}$ there exists $u=(u_1,u_2,...,u_N)\in U^\eta$ such that condition (\ref{condCi0}) holds for all $i\in [1;N]$ from which, $x_Q \rTo_{Q,\eta,c} x^+_Q$ is a transition of $S_{Q,\eta}$.
Conversely, consider any $x_Q \rTo_{Q,\eta} x^+_Q$. Hence, for all $i\in [1;N]$ there exists $u_i\in U^{\eta}_i$ such that  
condition (\ref{condCi0}) holds for all $i\in [1;N]$, implying by definition of $\rTo_{Q,\eta,c}$ that $x_Q \rTo_{Q,\eta} x^+_Q$ is a transition of $S^c_{Q,\eta}$. 
Thus, $\rTo_{Q,\eta}=\rTo_{Q,\eta,c}$ from which, $S_{Q,\eta}=S^c_{Q,\eta}$ and hence,  
$\Trim(S^c_{Q,\eta})=\Trim(S_{Q,\eta})$ which, by Corollaries \ref{ThPartDec} and \ref{ThPartCen}, implies the statement.
\end{proof}

This result shows that any word of the specification that can be enforced by the centralized controller $C_c$ can also be enforced by the decentralized controller $C$. The only difference is that local controllers $C_i$ need to agree in advance on which word to enforce since 
in a decentralized control architecture no communication among local controllers is allowed (see Remark \ref{remarknew}).

\section{Efficient controllers synthesis and computational complexity analysis} \label{sec6}
In this section we extend on--the--fly algorithms of \cite{PolaTAC12} to the 
synthesis of the decentralized controllers designed in Section \ref{sec4}. 
The on--the--fly procedure is reported in Algorithm \ref{alg}, where the main idea is to design controllers $C_i$ in (\ref{Ci}) without computing explicitly systems $S^{\eta} (\Sigma_i)$. Starting from system $S_Q$ in (\ref{specsys}), associated with the specification $L_Q$, Algorithm \ref{alg} returns as output, system $\Trim(S_{Q,\eta})$ and functions $h_{c,i}$, $i\in [1;N]$, on the basis of which, solution to Problem \ref{problem} is specified in (\ref{init}) and (\ref{deccontr}). It computes in line 6, for each transition $x_Q \rTo_Q x_Q^+$ and for each control system $\Sigma_i$, the set of control inputs $h_{c,i}(x_Q)$. If $h_{c,i}(x_Q)\neq \varnothing$, then transition $x_Q \rTo_Q x_Q^+$ can be matched by $\Sigma_i$ by picking any control input $u_i \in h_{c,i}(x_Q)$; in this case, 
$\mathcal{I}_i(x_Q \rTo_Q x_Q^+,\eta)$ is set in line 8 to $\tt{True}$. If each control system $\Sigma_i$ can match transition $x_Q \rTo_Q x_Q^+$, then resulting $\mathcal{I}(x_Q \rTo_Q x_Q^+,\eta)$ in line 11 evaluates as $\tt{True}$. 
System $S_{Q,\eta}$ can then be computed in line 13, by (\ref{SQeta}) that uses $\mathcal{I}$. $\Trim(S_{Q,\eta})$ is finally computed in line 14.  
Formal correctness of Algorithm \ref{alg} follows from the definition of the sets and operators involved. We also report Algorithm \ref{alg1} for  designing centralized controllers in Section \ref{sec5}, which follows the same reasoning of Algorithm \ref{alg}. \\
We conclude with a computational complexity analysis. Let $N_Q$ and $N_{U,i}$ be the cardinality of $\rTo_Q$ and of $\mathcal{U}_i$, respectively. 
It is readily seen that:

\begin{proposition}
\label{fff}
Space and time computational complexity $cc.dec$ of Algorithm \ref{alg} scales as $O(N_Q \sum_{i \in [1;N]} N_{U,i})$.
\end{proposition}

\begin{proposition}
\label{ooo}
Space and time computational complexity $cc.cen$ of Algorithm \ref{alg1} scales as $O(N_Q \prod_{i \in [1;N]} N_{U,i})$.
\end{proposition}

As a direct consequence of Propositions \ref{fff} and \ref{ooo}, when $N_Q$ does not depend on $N$, 
as for example in the case of motion planning types specifications, we get  
\[
cc.cen \sim O(2^N), 
\quad
cc.dec \sim O(N), 
\]
i.e. from \textit{exponential} complexity with $N$ in the centralized case, to \textit{linear} complexity with $N$ in the decentralized case. 

\begin{remark}
It is easy to see that lines 4--10 of Algorithm \ref{alg}, corresponding to evaluate indicators $\mathcal{I}_i(x_Q \rTo_Q x_Q^+,\eta)$ for all $i\in [1;N]$, can be implemented via $N$ computing units, which can work independently from each other, thus leading to a parallel computing architecture. This architecture allows reduction of the time computational complexity bound in Proposition \ref{fff} from $O(N_Q \sum_{i \in [1;N]} N_{U,i})$ to $O(N_Q \max_{i \in [1;N]} N_{U,i})$ which, when $N_Q$ is independent from $N$, yields a time computational complexity which is independent from $N$. Space computational complexity does not reduce in this case and hence, 
scales as $O(N)$. 
\end{remark}

In conclusion, centralized and decentralized control architectures allow enforcing the same part of the specification $L_Q$, with the disadvantage in the decentralized case to agree in advance on which word to enforce, but with advantages in terms of computational complexity.  

\incmargin{1em}
\restylealgo{boxed}\linesnumbered
\begin{algorithm}[t]
\label{alg}
\SetLine
\caption{Decentralized local controllers design.}
\textbf{input:} $S_Q=(X_Q,X_{Q,0},U_Q,\rTo_{Q},X_{Q,m},\mathbb{R}^n,H_Q)$\;
\ForEach {$x_Q \rTo_Q x_Q^+$}
{
\textbf{set} $\mathcal{I}(x_Q \rTo_Q x_Q^+,\eta):=\tt{True}$\; 
\ForEach {$i \in [1;N]$}
{
\textbf{set} $\mathcal{I}_i(x_Q \rTo_Q x_Q^+,\eta):=\tt{False}$\; 
\textbf{compute} the set $h_{c,i}(x_Q)$ of all $u_i\in U^{\eta}_i$
satisfying (\ref{condCi0})\;
\If{$h_{c,i}(x_Q) \neq \varnothing$}
{
\textbf{set} $\mathcal{I}_i(x_Q \rTo_Q x_Q^+,\eta):=\tt{True}$\; 
}
}
\textbf{set} $\mathcal{I}(x_Q \rTo_Q x_Q^+,\eta):=\wedge_{i \in [1;N]} \mathcal{I}_i(x_Q \rTo_Q x_Q^+,\eta)$\;
}
\textbf{compute} $S_{Q,\eta}$ in (\ref{SQeta})\;
\textbf{compute} $\Trim(S_{Q,\eta})$\;
\textbf{output:} $\Trim(S_{Q,\eta})$ and $h_{c,i}$, $i\in [1;N]$\;
\end{algorithm}
\decmargin{1em}

\incmargin{1em}
\restylealgo{boxed}\linesnumbered
\begin{algorithm}
\label{alg1}
\SetLine
\caption{Centralized controller design.}
\textbf{input:} $S_Q=(X_Q,X_{Q,0},U_Q,\rTo_{Q},X_{Q,m},\mathbb{R}^n,H_Q)$\;
\ForEach {$x_Q \rTo_Q x_Q^+$}
{
\textbf{set} $\mathcal{I}_c(x_Q \rTo_Q x_Q^+,\eta):=\tt{False}$\; 
\textbf{compute} the set $h_{c}(x_Q)$ of all $u=(u_1,u_2,...,u_N)\in U^{\eta}$
satisfying (\ref{condCi0}) for all $i\in [1;N]$\;
\If{$h_{c}(x_Q) \neq \varnothing$}
{
\textbf{set} $\mathcal{I}_c (x_Q \rTo_Q x_Q^+,\eta):=\tt{True}$\; 
}
}
\textbf{compute} $S^c_{Q,\eta}$ in (\ref{SQetac})\;
\textbf{compute} $\Trim(S^c_{Q,\eta})$\;
\textbf{output:} $\Trim(S^c_{Q,\eta})$ and $h_{c}$\;
\end{algorithm}
\decmargin{1em}

\section{An illustrative example}\label{sec7}
We consider the problem of regulating the temperature in a circular building composed of $N\geq 3$ rooms, each one equipped with a heater. This example set--up is adapted  from \cite{GirardTAC16}. 
The evolution in time of the temperature $\mathbf{T}_i(t)$ of room $i$ with $i\in [1;N]$ is described by control systems $\Sigma_i$:
\begin{equation}
\label{sysexample2}
\begin{array}
{rcl}
\mathbf{T}_i (t+1) & = & \mathbf{T}_{i}(t) + \alpha(\mathbf{T}_{i+1}(t)+\mathbf{T}_{i-1}(t)-2\mathbf{T}_{i}(t))\\
                   &   & + \beta(T_e - \mathbf{T}_{i}(t))+ \gamma (T_h - \mathbf{T}_{i}(t))\mathbf{u}_i(t),
\end{array}
\end{equation}
where $\mathbf{T}_{i+1}(t)$ and $\mathbf{T}_{i-1}(t)$ are the temperature in Celsius degrees at (discrete) time $t$ of rooms $i+1$ and $i-1$, respectively (here and in the sequel indices $0$ and $N+1$ correspond to $N$ and $1$, respectively); $T_e$ is the temperature of the external environment of the building; $T_h$ is the temperature of the heater; $\alpha \in \mathbb{R}^+$ is the conduction factor between rooms $i \pm 1$ and room $i$; $\beta \in \mathbb{R}^+$ is the conduction factor between the external environment and room $i$; $\gamma \in \mathbb{R}^+$ is the conduction factor between the heater and room $i$. Control inputs $\mathbf{u}_i(t)$ at time $t$ assume values in $\mathcal{U}_i = (0.025\mathbb{Z})\cap [0,1]$. 
The specification requires $\mathbf{T}_i(t)$ to follow Table \ref{TabSpec} up to an accuracy $\theta=0.5$. 
We start by checking Assumption \ref{A1}. 
Define 
$
A=\max\{\vert 1- 2\alpha-\beta-\gamma \vert,\vert 1-2\alpha-\beta \vert \} + 2 \alpha 
$. 
Network of control systems $\Sigma_i$ in (\ref{sysexample2}) admits the following $\delta$--GAS Lyapunov function 
\begin{equation}
\label{Vexample}
V(x,x')= \Vert x - x' \Vert ,
\end{equation}
for any $x=(\mathbf{T}_{1},\mathbf{T}_{2},...,\mathbf{T}_{N}) \in \mathbb{R}^N$, $x'=(\mathbf{T}'_{1},\mathbf{T}'_{2},...,\mathbf{T}'_{N})\in \mathbb{R}^N$, with $\mathcal{K}_{\infty}$ functions
\begin{equation}
\label{fexample}
\underline{\alpha}(s)=\overline{\alpha}(s)=s,\quad
\rho(s)=(1-A) s, \quad s\in\mathbb{R}^+_0 ,
\end{equation}
provided that
\begin{equation}
\label{condex1}
A<1.
\end{equation}
Bounding function $\sigma$ of $V$ as in (\ref{ineqnew}) can be chosen as $\sigma(s)=s,s\in\mathbb{R}^+_0$. 
In this example we pick a uniform quantization $\eta\in\mathbb{R}_N^+$, i.e. $\eta(i)=\eta(j)$ for all $i,j \in [1;N]$ and for ease of notation we use $\eta$ instead of $\eta(i)$ or $\Vert \eta \Vert$. Inequality in (\ref{statem}) rewrites as 
$\eta \leq \mu \min \{(1-A),1\}$. 
While in concrete applications, parameters $\alpha$, $\beta$ and $\gamma$ need to be identified, in the sequel we choose $\alpha=0.45$, $\beta=0.045$, $\gamma=0.09$, corresponding to the Euler discretization of the model proposed in \cite{GirardTAC16} with sampling time $\tau=9$. (Larger values of the sampling time lead to instability of the discretized system.) 
We further set $T_h=50$ and $T_e=-1$ (L'Aquila is a cold city). We get $A=0.955$ that satisfies (\ref{condex1}) so that Assumption \ref{A1} holds. 
Specification in Table \ref{TabSpec} can be formalized by means of the regular expression $\mathbf{q}\mathbf{q}^\ast$ where 
\begin{equation}
\label{specc}
\begin{array}
{rl}
\mathbf{q} = & (19, 18,...,18)
  (19, 18.5,...,18.5)
 (19, 19,...,19)\\
 & (19, 19.5,...,19.5)
  (19, 20,...,20)
 (19, 20,...,20)\\
 & (19, 20,...,20)
(19, 19.5,...,19.5)
  (19, 19,...,19)\\
 & (20, 18.5,...,18.5)
(20, 18.25,...,18.25)
(19, 18,...,18).
\end{array}
\end{equation}
We get $Y_Q=\{18,18.25,18.5,19,19.5,20\}$. 
For the desired accuracy $\theta=0.5$ we can pick $\mu=0.5$ and also $\eta=0.0225$ which satisfy (\ref{statem}) and (\ref{statem3}). By this choice of $\eta$, we get $Y_Q\subset X^{\eta}$ by which, we can apply Corollary \ref{CoromaindecBis}. 
Algorithm \ref{alg} returns local controllers $C_i$ in Table \ref{TabDecContr} and $\Trim(S_{Q,\eta})=S_Q$ where $S_Q$ marks $\mathbf{q}\mathbf{q}^*$. Controllers $C_i$ for $i\in [2;N]$ are of two types: controllers $C_2=C_N$ and controllers $C_i$, $i\in [3;N-1]$ that correspond to rooms with neighboring rooms requested to follow different temperature schedules, see Table \ref{TabSpec}.
Sets $\mathcal{X}_0$ and $\mathcal{X}_f$ involved in Problem \ref{problem} result in:
\begin{equation}
\label{initfinal}
\mathcal{X}_0=\mathcal{X}_f=[18.5,19.5] \times (\bigtimes_{i\in[2;N]} [17.5,18.5]).
\end{equation}
Table \ref{TabDecContr} and (\ref{initfinal}) fully specify the solution to Problem \ref{problem} which has been solved for an arbitrarily large number $N$ of rooms.  
We report in Table \ref{TabSim} the results of the simulations on the controlled system. By comparing Tables \ref{TabSpec} and \ref{TabSim}, and by recalling the accuracy $\theta=0.5$ chosen, it is readily seen that the specification is met.   
Time of computation of Algorithm \ref{alg} is $0.1563$s, without using parallel computing architectures. 
We solved the same problem for the case of only $N=4$ rooms by using the centralized approach in Algorithm \ref{alg1}. We obtained $\Trim(S^c_{Q,\eta})=\Trim(S_{Q,\eta})$, in accordance with Theorem \ref{thconnex}. Time of computation is $163.6304$s. 
Computations have been performed on a Lenovo IP YOGA 3 PRO 8GB 512SSD.

\section{Conclusions} \label{sec8}
In this paper we proposed decentralized control architectures for enforcing regular language specifications on networks of discrete--time nonlinear control systems, within any desired accuracy. The approach taken was based on the use of symbolic models and on--the--fly inspired algorithms. A comparison with centralized control architectures was formally discussed which included also a computational complexity analysis. 
An illustrative example was presented, which showed the applicability and effectiveness of our results. 

\medskip      

\textit{Aknowledgement: }
We thank Luca Schenato for fruitful inputs on decentralized control of dynamical systems.

\begin{table}
\begin{center}
\begin{tabular}
[c]{lll}
\hline
t mod($12$)  & $\mathbf{T}_1(t)$ & $\mathbf{T}_i(t)$, $i\in [2;N]$\\ \hline
0            &   19     & 18     \\ \hline   
1            &   19     & 18.5   \\ \hline
2            &   19     & 19     \\ \hline
3            &   19     & 19.5   \\ \hline
4            &   19     & 20     \\ \hline
5            &   19     & 20     \\ \hline
6            &   19     & 20     \\ \hline
7            &   19     & 19.5   \\ \hline
8            &   19     & 19     \\ \hline
9            &   19     & 18.5   \\ \hline
10           &   19     & 18.25   \\ \hline
11           &   19     & 18   \\ \hline
\end{tabular}
\caption{Specification $L_Q$.}
\label{TabSpec}
\end{center}
\end{table}

\begin{table}
\begin{center}
\begin{tabular}
[c]{llll}
\hline
t mod($12$) & $C_1$     & $C_2 = C_N$          & $C_i$, $i\in [3;N-1]$ \\ \hline
0  &  \{0.65\}          &   \{0.45\}       & \{0.6\}              \\ \hline   
1  &  \{0.475\}         &   \{0.55\}     & \{0.625\}    \\ \hline
2  &  \{0.325\}         &   \{0.65\}     & \{0.65\}          \\ \hline
3  &  \{0.15\}          &   \{0.75\}     & \{0.65\}     \\ \hline
4  &  \{0\}             &   \{0.525\}     & \{0.35\}            \\ \hline
5  &  \{0\}             &   \{0.525\}     & \{0.35\}      \\ \hline
7  &  \{0\}             &   \{0.175\}     & \{0.025\} \\ \hline
8  &  \{0.15\}          &   \{0.1\}      & \{0\} \\ \hline
9  &  \{0.325\}         &   \{0\}      & \{0\} \\ \hline
6  &  \{0.475\}         &   \{0.075\}     & \{0.15\} \\\hline
10 &  \{0.55\}          &   \{0.025\}         & \{0.15\} \\ \hline
11 &  \{0.65\}          &   \{0.15\}         & \{0.30\} \\ \hline
\end{tabular}
\caption{Local controllers $C_i$.}
\label{TabDecContr}
\end{center}
\end{table}

\begin{table}
\begin{center}
\begin{tabular}
[c]{llll}
\hline
t mod($12$) & $\mathbf{T}_1(t)$       & $\mathbf{T}_2(t)=\mathbf{T}_N(t)$          & $\mathbf{T}_i(t)$, $i\in [3;N-1]$ \\ \hline
0  &  19.5000             &   18.5000        & 17.5000 \\ \hline   
1  &  18.9788             &   18.8462        & 18.0125 \\ \hline
2  &  18.7329             &   19.2453        & 18.5368 \\ \hline
3  &  18.6496             &   19.6773        & 19.0709 \\ \hline
4  &  18.6042             &   20.1282        & 19.5744 \\ \hline
5  &  18.5992             &   20.1021        & 19.5924 \\ \hline
6  &  18.6058             &   20.0838        & 19.6098 \\ \hline
7  &  18.6176             &   19.5475        & 19.1325 \\ \hline
8  &  18.6200             &   19.0492        & 18.6292 \\ \hline
9  &  18.6440             &   18.5357        & 18.1385 \\ \hline
10 &  18.6448             &   18.2824        & 17.8990 \\ \hline
11 &  18.6431             &   18.0186        & 17.9080 \\ \hline
\end{tabular}
\caption{Simulation results.}
\label{TabSim}
\end{center}
\end{table}

\bibliographystyle{plain}
\bibliography{biblio2}

\section{Appendix}

\subsection{Notation}
A directed graph $\mathcal{G}$ is a pair $(\mathcal{V},\mathcal{E})$ where $\mathcal{V}$ is the set of vertices and $\mathcal{E}\subseteq \mathcal{V} \times \mathcal{V}$ is the set of edges. 
Symbol $\wedge$ denotes the logical conjunction.  
Given two sets $X$ and $Y$ and relation $\mathcal{R}\subseteq X\times Y$, symbol $\mathcal{R}^{-1}$ denotes the inverse relation of $\mathcal{R}$, i.e.
$\mathcal{R}^{-1}=\{(y,x)\in Y\times X|( x,y)\in \mathcal{R}\}$. Given $X'\subseteq X$ and $Y'\subseteq Y$, we denote $\mathcal{R}(X')=\{y\in Y | \exists x\in X' \text{ s.t. } (x,y)\in \mathcal{R}\}$ and $\mathcal{R}^{-1}(Y')=\{x\in X | \exists y\in Y' \text{ s.t. }  (x,y)\in \mathcal{R}\}$. Symbols $\mathbb{N}_0$, $\mathbb{Z}$, $\mathbb{R}$, $\mathbb{R}^{+}$ and $\mathbb{R}_{0}^{+}$ denote the set of nonnegative integer, integer, real, positive real, and nonnegative real numbers, respectively. Symbol $\mathbb{R}^{+}_{n}$ denotes the positive orthant of $\mathbb{R}^{n}$. 
The symbol $0_n$ denotes the origin of $\mathbb{R}^n$. 
Given $n\in \mathbb{N}_0$ and $n>0$, symbol $[1;n]$ denotes $\{1,2,...,n\}$. 
Given $x\in\mathbb{R}^{n}$, symbol $x(i)$ denotes the $i$--th element of $x$ and $\Vert x\Vert$ the infinity norm of $x$. Given $a\in\mathbb{R}$ and $X\subseteq \mathbb{R}^{n}$, symbol $aX$ denotes the set $\{y\in\mathbb{R}^{n}| \exists x\in X \text{ s.t. } y=ax\}$. 
Given $\theta\in\mathbb{R}^+$ and $x\in\mathbb{R}^n$, we define $\mathcal{B}_{[\theta[}(x)=\{y\in\mathbb{R}^{n}| y(i) \in [x(i)-\theta,x(i)+\theta[, i\in[1;n]\}$.
Note that for any $\theta\in\mathbb{R}^+$, $\{\mathcal{B}_{[\theta
[}(x)\}_{x\in 2\theta \, \mathbb{Z}^n}$ is a partition of $\mathbb{R}^n$. 
Given $z\in  \mathbb{R}^{n}$, symbol $[z]_{\theta}$ denotes the unique vector in $\theta\,\mathbb{Z}^{n}$ such that $z\in \mathcal{B}_{[\theta/2
[}([z]_{\theta})$. 
%
A continuous function \mbox{$\rho:\mathbb{R}_{0}^{+}\rightarrow\mathbb{R}_{0}^{+}$} is said to belong to class $\mathcal{K}$ if it is strictly increasing and \mbox{$\rho(0)=0$}; function $\rho$ is said to belong to class $\mathcal{K}_{\infty}$ if \mbox{$\rho\in\mathcal{K}$} and $\rho(r)\rightarrow\infty$ as $r\rightarrow\infty$. 

\subsection{Systems, regular languages and approximate bisimulation}\label{sec:ApproxEquiv}
We recall from e.g. \cite{Cassandras} some notions on formal language theory. Let $Y$ be a finite set representing the alphabet. A word over $Y$ is a finite sequence $y_{1} \, y_{2} \, ... \, y_{l}$ of symbols in $Y$. The empty word is denoted by $\varepsilon$. 
The symbol $Y^\ast$ denotes the Kleene closure of $Y$, that is the collection of all words over $Y$ including $\varepsilon$. A language $L$ over $Y$ is a subset of $Y^\ast$. 
We now recall the notion of system:

\begin{definition}
\label{systems}
A system is a tuple $S=(X,X_0,U,\rTo,$ $X_m,Y,H)$, consisting of a set of states $X$, a set of initial states $X_0 \subseteq X$, a set of inputs $U$, a transition relation $\rTo \subseteq X\times U\times X$, a set of marked states $X_m \subseteq X$, a set of outputs $Y$ and an output function $H:X\rightarrow Y$.
\end{definition}

The above definition slightly extends the one of \cite{paulo} to systems with marked states. 
A transition $(x,u,x^{\prime})\in\rTo$ of $S$ is denoted by $x\rTo^{u}x^{\prime}$. System $S$ is empty if $X_0=\varnothing$. 
The evolution of systems is captured by the notions of state, input and output runs. Given a sequence of transitions of $S$
\begin{equation}
\label{seqtrans}
x_0 \rTo^{u_0} x_1 \rTo^{u_1} \,{...}\, \rTo^{u_{l-1}} x_l
\end{equation}
with $x_0 \in X_0$, the sequences 
\begin{eqnarray}
&& r_X: \, x_0 \, x_1 \, ... \, x_l,\notag\\
&& r_U: \, u_0 \, u_1 \, ... \, u_{l-1}, \label{inputrun}\\
&& r_Y: H(x_0) \, H(x_1) \, ... \, H(x_l), \label{outputrun}
\end{eqnarray}
are called a \textit{state run}, an \textit{input run} and an \textit{output run} of $S$, respectively. 
System $S$ is said to be \textit{symbolic/finite} if $X$ and $U$ are finite sets, \textit{metric} if $Y$ is equipped with a metric $\mathbf{d}:Y\times Y\rightarrow\mathbb{R}_{0}^{+}$, \textit{deterministic} if for any $x\in X$ and $u\in U$ there exists at most one transition $x \rTo^u x^+$ and \textit{nondeterministic}, otherwise. System $S$ is said \textit{nonblocking} if 
for any transitions sequence (\ref{seqtrans}) of $S$ with $x_0\in X_0$ either $x_l \in X_m$ or there exists a continuation 
$
x_0 \rTo^{u_0} x_1 \rTo^{u_1} \,{...}\, \rTo^{u_{l-1}} x_l \rTo^{u_l} \,{...}\, \rTo^{u_{l'-1}} x_{l'}
$
of it such that $x_{l'} \in X_m$, 
and \textit{blocking}, otherwise. 
The \textit{input language} (resp. \textit{output language}) of $S$, denoted $\mathcal{L}^u(S)$ (resp. $\mathcal{L}^y(S)$), is the collection of all its input runs (resp. output runs). The \textit{marked input language} (resp. \textit{marked output language}) of $S$, denoted as $\mathcal{L}_m^u(S)$ (resp. $\mathcal{L}_m^y(S)$), is the collection of all input runs $r_U$ in (\ref{inputrun}) (resp. output runs $r_Y$ in (\ref{outputrun})) such that the corresponding transitions sequence in (\ref{seqtrans}) is with ending state $x_l\in X_m$. A language $L$ over a finite set $U$ is said \textit{regular} if there exists a symbolic system $S$ with input set $U$ such that $L=\mathcal{L}_m^u (S)$.  
We also recall some unary operations on systems naturally adapted from the ones given for DES \cite{Cassandras}. 
A system $S'=(X',X'_0,U',\rTo',X'_m,Y',H')$ is said to be a \textit{subsystem} of $S=(X,X_0,U,\rTo,X_m,Y,H)$, denoted $S' \sqsubseteq S$, if $X'\subseteq X$, $X'_0\subseteq X_0$, $U'\subseteq U$, $\rTo' \subseteq \rTo$, $X'_m \subseteq X_m$, $Y' \subseteq Y$ and $H'(x)=H(x)$ for all $x\in X'$. 
The accessible part of $S$, denoted $\Ac(S)$, is the unique maximal\footnote{Here, maximality is with respect to the pre--order naturally induced by the binary operator $\sqsubseteq$.} subsystem $S'$ of $S$ such that for any state $x'$ of $S'$ there exists a state run of $S'$ ending in $x'$. 
By definition, 
if $S$ is nonempty, $\Ac(S)$ is accessible. 
The co--accessible part of $S$, denoted $\CoAc(S)$, is the unique maximal$^2$ subsystem $S'$ of $S$ such that for any state $x'\in X'$ there exists a transition sequence of $S'$ starting from $x'$ and ending in a marked state of $S'$. 
By definition, 
$\CoAc(S)$ if not empty, is nonblocking. 
The trim of $S$, denoted $\Trim(S)$, is defined as $\Trim(S)=\CoAc(\Ac(S))=\Ac(\CoAc(S))$. 
By definition, 
$\Trim(S)$, if not empty, is accessible and nonblocking. 
We conclude by recalling 
some notions related to systems' simulation and bisimulation:

\begin{definition}
\label{ASR}
\cite{Borri16} 
Let $S_{i}=(X_{i},X_{0,i},U_{i},\rTo_{i},X_{m,i},Y_{i},$ $H_{i})$ ($i=1,2$) be metric systems with the same input set $U_{1}=U_{2}$, output sets $Y_{1}=Y_{2}$ and metric $\mathbf{d}$, and let $\mu\in\mathbb{R}^{+}_{0}$ be a given accuracy. A relation $\mathcal{R}\subseteq X_{1}\times X_{2}$ is said a strong $\mu$-approximate simulation relation from $S_{1}$ to $S_{2}$ if it enjoys the following conditions:
\begin{itemize}
\item [(i)] $\forall x_{1}\in X_{0,1}$ $\exists x_{2}\in X_{0,2}$ such that $(x_{1},x_{2})\in \mathcal{R}$;
\item [(ii)] $\forall x_{1}\in X_{m,1}$ $\exists x_{2}\in X_{m,2}$ such that $(x_{1},x_{2})\in \mathcal{R}$;
\item [(iii)] $\forall (x_{1},x_{2})\in \mathcal{R}$, \mbox{$\mathbf{d}(H_{1}(x_{1}),H_{2}(x_{2}))\leq\mu$};
\item[(iv)] $\forall (x_{1},x_{2})\in \mathcal{R}$ if \mbox{$x_{1}\rTo_{1}^{u}x'_{1}$} then there exists \mbox{$x_{2}\rTo_{2}^{u}x'_{2}$} such that $(x^{\prime}_{1},x^{\prime}_{2})\in \mathcal{R}$.
\end{itemize}
Relation $\mathcal{R}$ is a strong $\mu$-approximate bisimulation relation between $S_{1}$ and $S_{2}$ if $\mathcal{R}$ is a strong $\mu$-approximate simulation relation from $S_{1}$ to $S_{2}$ and $\mathcal{R}^{-1}$ is a strong $\mu$-approximate simulation relation from $S_{2}$ to $S_{1}$. Systems $S_{1}$ and $S_{2}$ are strongly $\mu$-bisimilar, denoted \mbox{$S_{1}\cong_{\mu} S_{2}$}, if there exists a strong $\mu$-approximate bisimulation relation $\mathcal{R}$ between $S_{1}$ and $S_{2}$. 
\end{definition}

The above notion requires stronger conditions than approximate (bi)simulation of \cite{AB-TAC07} that allows transitions in condition (iv) with possibly different control labels. 

\end{document}